\documentclass[12pt]{amsart}

\usepackage{amssymb,amsmath,mathrsfs}
\usepackage[ps2pdf,colorlinks=true,urlcolor=blue,
citecolor=red,linkcolor=blue,linktocpage,pdfpagelabels,bookmarksnumbered,bookmarksopen]{hyperref}

\usepackage[left=2.82cm,right=2.82cm,top=2.7cm,bottom=2.7cm]{geometry}

\newcommand{\R}{\mathbb{R}}
\newcommand{\D}{\mathbb{D}}
\newcommand{\C}{\mathbb{C}}
\newcommand{\E}{{\mathcal E}}
\renewcommand{\ge}{\varepsilon}

\newcommand{\im}{{\Im}}
\newcommand{\re}{{\Re}}
\newcommand{\M}{\mathcal{M}}
\newcommand{\RN}{\mathbb{R}^3}
 \newcommand{\Z}{{ \it Z}}
\newcommand{\iu}{{\rm i}}

 \let\e=\ge

\newcommand{\p}{\varphi}

 \newcommand{\intr}{\int_{\mathbb{R}^3}}

\numberwithin{equation}{section}
\newtheorem{theorem}{Theorem}[section]
\newtheorem{proposition}[theorem]{Proposition}
\newtheorem{corollary}[theorem]{Corollary}
\newtheorem{lemma}[theorem]{Lemma}
\newtheorem{definition}[theorem]{Definition}
\theoremstyle{definition}
\newtheorem{remark}[theorem]{Remark}

\newcommand{\brm}{\begin{remark}\rm}
\newcommand{\erm}{\end{remark}}
\newcommand{\bte}{\begin{theorem}}
\newcommand{\ete}{\end{theorem}}
\newcommand{\bpr}{\begin{proposition}}
\newcommand{\epr}{\end{proposition}}
\newcommand{\ble}{\begin{lemma}}
\newcommand{\ele}{\end{lemma}}
\newcommand{\beq}{\begin{equation}}
\newcommand{\eeq}{\end{equation}}
\newcommand{\bdm}{\begin{displaymath}}
\newcommand{\edm}{\end{displaymath}}
\newcommand{\W}{{\mathcal W}}
\numberwithin{equation}{section}

\let\eps=\varepsilon 

\title[NLS with magnetic field and Hartree-type nonlinearities]{Semiclassical limit
for Schr\"{o}dinger equations with magnetic field and Hartree-type nonlinearities}

\author[S.\ Cingolani]{Silvia Cingolani}
\author[S.\ Secchi]{Simone Secchi}
\author[M.\ Squassina]{Marco Squassina}

\address{Dipartimento di Matematica,
Politecnico di Bari,
Via Orabona 4, I-70125 Bari, Italy.}
\email{s.cingolani@poliba.it}

\address{Dipartimento di Matematica ed Applicazioni,
Universit\`a di Milano Bicocca,
Via R. Cozzi 53 edificio U5,  I-20125 Milano, Italy.}
\email{simone.secchi@unimib.it}

\address{Dipartimento di Informatica,
Universit\`a di Verona,
C\`a Vignal 2, Strada Le Grazie 15, I-37134 Verona, Italy.}
\email{marco.squassina@univr.it}

\thanks{The research of the first and the third
  author was partially supported by the MIUR national research project
  \textit{Variational and Topological Methods in the Study of Nonlinear
    Phenomena}, PRIN 2007. The research of the second author was partially supported by the MIUR national
  research project \textit{Variational methods and nonlinear
    differential equations}, PRIN 2006.}


\begin{document}
\date{\today}

\begin{abstract}
  The semi-classical regime of standing wave solutions of a
  Schr\"odinger equation in presence of non-constant electric and
  magnetic potentials is studied in the case of non-local
  nonlinearities of Hartree type. It is show that there exists a
  family of solutions having multiple concentration regions which are located
  around the minimum points of the electric potential.
\end{abstract}
\maketitle

\medskip
\begin{center}
\begin{minipage}{11cm}
\footnotesize
\tableofcontents
\end{minipage}
\end{center}

\medskip

\section{Introduction and main result}

\subsection{Introduction}
Some years ago, Penrose derived in~\cite{p} a system of nonlinear equations by coupling the linear Schr\"odinger equation of quantum mechanics with Newton's gravitational law. Roughly speaking, a point mass interacts with a density of matter described by the square of the wave function that solves the Schr\"odinger equation.
If $m$ is the mass of the point, this interaction leads to the system in $\R^3$
\begin{equation}
    \label{sys:s-n}
\begin{cases}
\frac{\hbar^2}{2m} \Delta \psi - V(x) \psi + U \psi = 0, \\
\Delta U + 4 \pi \gamma |\psi|^2 =0,
\end{cases}
\end{equation}
where $\psi$ is the wave function, $U$ the gravitational potential energy,
$V$ a given Schr\"odinger potential, $\hbar$ the Planck constant and $\gamma = G m^2$,
$G$ being Newton's constant of gravitation.
Notice that, by means of the scaling
\begin{equation*}
\psi(x)=\frac{1}{\hbar} \frac{\hat{\psi}(x)}{\sqrt{8 \pi \gamma m}},
\quad V(x)=\frac{1}{2m} \hat{V}(x),
\quad U(x)=\frac{1}{2m}\hat{U}(x)
\end{equation*}
system~\eqref{sys:s-n} can be written, maintaining the original notations, as
\begin{equation}
    \label{sys:s-n-red}
\begin{cases}
\hbar^2 \Delta \psi - V(x) \psi + U \psi = 0, \\
\hbar^2 \Delta U + |\psi|^2 =0.
\end{cases}
\end{equation}
The second equation in~\eqref{sys:s-n-red} can be explicitly solved with
respect to $U$, so that the system turns into
the single nonlocal equation
\begin{equation} \label{eq:1.3}
\hbar^2 \Delta \psi - V(x)\psi + \frac{1}{4\pi \hbar^2} \Big( \int_{\R^3}
\frac{|\psi(\xi)|^2}{|x-\xi|} d\xi \Big)\psi =0\quad\text{in $\R^3$}.
\end{equation}
The Coulomb type convolution potential $W(x)=|x|^{-1}$ in $\R^3$  is also involved in  various physical
applications such as electromagnetic waves in a Kerr medium (in nonlinear optics), surface gravity
waves (in hydrodynamics) as well as ground states solutions (in quantum mechanical
systems).\ See for instance~\cite{abpr} for further details and~\cite{erdosyau} for the
derivation of these equations from a many-body Coulomb system.
\vskip2pt
In the present paper, we will study the semiclassical regime (namely the existence and asymptotic
behavior of solutions as $\hbar \to 0$) for a more general equation having a similar structure.
Taking $\varepsilon$ in place of $\hbar$, our model will be written as
\begin{equation}
\label{eq:Hartree}
\Big( \frac{\ge}{\iu} \nabla - A(x) \Big)^2 u + V(x) u = \frac{1}{\ge^{2}}\left( W * |u|^2 \right)u
\quad\text{in $\R^3$},
\end{equation}
in $\mathbb{R}^3$, where the convolution kernel
$W:\R^3\setminus\{0\} \to (0,\infty)$ is an even smooth
kernel, homogeneous of degree $-1$ and we denote by $\iu$ the imaginary unit. The choice of
$W(x)=|x|^{-1}$ recovers~\eqref{eq:1.3}. Equation~\eqref{eq:Hartree} is equivalent to
\begin{equation}
  \label{eq:Hartree-scaled}
  \Big( \frac{1}{\iu} \nabla - A_\varepsilon(x) \Big)^2 u + V_\varepsilon(x) u = \left( W * |u|^2 \right)u
\quad\text{in $\R^3$},
\end{equation}
where we have set $A_\varepsilon (x) = A(\varepsilon x)$ and $V_\varepsilon(x)=V(\varepsilon x)$.
\vskip2pt
The vector-valued field $A$ represents a given external magnetic potential, and forces the solutions to be, in general, complex-valued (see~\cite{cs-jmp} and references therein). To the best of our knowledge, in this framework, no previous result involving the electromagnetic field can be found in the literature. On the other hand, when $A \equiv 0$, it is known that solutions have a constant phase, so that it is not a restriction to look for real-valued solutions. In this simpler situation, we recall the results contained in \cite{moroz,mpt}, stating that at fixed $\hbar=\varepsilon$ the system~\eqref{sys:s-n-red} can be uniquely solved by radially symmetric functions. Moreover these solutions decay exponentially fast at infinity together with their first derivatives. The mere existence of one solution can be traced backed to the paper~\cite{L}.
\vskip2pt
Later on, Wei and Winter proposed in~\cite{wei} a deeper study of the multi-bumps solutions to the same system, and proved an existence result that can be summarized as follows: if $k\geq 1$ and
$P_1,\ldots,P_k\in \mathbb{R}^3$ are given non-degenerate critical points of $V$ (but local extrema are also included without any further requirements), then multi-bump solutions $\psi_\hbar$ exist that concentrate at these points when $\hbar \to 0$.
A similar equation is also studied in~\cite{macri}, where multi-bump solutions are found by some finite-dimensional reduction.
The main result about existence leans on some {\em non-degeneracy} assumption on the solutions of a limiting problem, which was actually proved in~\cite[Theorem 3.1]{wei} only in the particular case $W(x)=|x|^{-1}$ in $\R^3$. Moreover, the equation investigated in~\cite{macri} cannot be deduced from a singularly perturbed problem like~\eqref{eq:1.3},
because the terms do not scale coherently.
\vskip2pt
 For precise references to some classical works (well-posedness,
regularity, long-term behaviour) related to
the nonlinear Schr\"odinger equation with Hartree nonlinearity for Coulomb potential and $A=0$, we refer to
\cite[p.66]{sulemsulem}. We would also like
to mention the work of Carles et al.~\cite{carlesmauser}.

\subsection{Statement of the main result}
We shall study equation~\eqref{eq:Hartree-scaled} by exploiting a penalization technique
which was recently developed in~\cite{cjs}, whose main idea is searching for solutions in
a suitable class of functions whose location and shape is the one expected for the solution itself.
This approach seems appropriate, since it does not need very strong knowledge of the \emph{limiting problem}~\eqref{eq:limiting} introduced in the next section. In particular, for a general convolution kernel $W$, we still do not know if its solutions are non-degenerate. In order to state our main result (as well as the technical lemma contained in Section 2 and 3),
the following conditions will be retained:
\vskip5pt
\begin{description}
\item[(A1)] $A:\R^3\to\R^3$ is of class $C^1$.
\item[(V1)] $V:\R^3\to\R$\, is a continuous function such that
$$
0 \le V_0 = \inf_{x \in \R^3} V(x),\qquad
\liminf_{|x| \to \infty}V(x) > 0.
$$
\item[(V2)] There exist bounded disjoint open sets $O^1,\ldots,O^k$ such that
\[
0 < m_i = \inf_{x \in O^i} V(x) < \min_{x \in \partial O^i} V(x),\,\,\quad i =1,\dots ,k.
\]

\item[(W)] $W \colon \mathbb{R}^3 \setminus \{0\} \to (0,\infty)$ is an even function of class~$C^1$ such that $W(\lambda x)=\lambda^{-1} W(x)$ for any $\lambda>0$ and $x \neq 0$.
\end{description}
\vskip6pt

Convolution kernels such as $W(x)= x^2_i /|x|^3$, for $x\in\R^3\setminus\{0\}$ or, more generally,
$W(x)=W_1(x)/W_2(x)$, for $x\in\R^3\setminus\{0\}$, where $W_1,W_2$ are positive, even and
(respectively) homogeneous of degree $m$ and $m+1$ satisfy $(W)$.
\vskip3pt

For each $i \in \{1,\ldots,k\},$ we define
\[
\M^i = \{x \in O^i: V(x) = m_i\}
\]
and $\Z = \{x \in \RN : V(x) = 0\}$ and $m=\min_{i \in \{1, \ldots, k\}}m_i$.
By~\textbf{(V1)} we can fix $\widetilde m>0$ with
\begin{equation*}
    \widetilde m <\min \Big\{ m, \ \liminf_{|x|\to\infty} V(x) \Big\}
\end{equation*}
and define ${\tilde V}_\e(x) = \max\{ \widetilde m,V_\e(x)\}.$ Let
$H_\e$ be the Hilbert space defined by the completion of
$C_0^\infty(\R^3, \C)$ under the scalar product
\begin{equation*}
\langle u,v \rangle_{\e} = \re
  \int_{\R^3} \Big(\frac{1}{\iu}\nabla u - A_\e (x)u \Big)
  \Big(\overline{\frac{1}{\iu}\nabla v- A_\e (x)v}\Big) + {\tilde
    V}_\e(x) u \overline{v}  dx
\end{equation*}
and $\| \cdot  \|_\e  $  the associated norm.
\vskip8pt

The main result of the paper is the following
\begin{theorem}
    \label{main}
  Suppose that \textbf{(A)}, \textbf{(V1-2)} and
  \textbf{(W)} hold. Then for any $\e>0$ sufficiently small, there
  exists a solution $u_\e \in H_\e $ of equation $\eqref{eq:Hartree-scaled}$ such that
  $|u_\e|$ has $k$ local maximum points $x_\e^i \in O^i$ satisfying
\[
\lim_{\e \to 0}\max_{i=1,\dots,k}\operatorname{dist}(\e x^i_\e,\M^i)
= 0,
\]
and for which
\[
|u_{\e}(x)| \leq C_1 \exp \big\{-C_2 \, \min_{i=1,...,k}|x- x_\ge^i|\big\},
\]
for some positive constants $C_1$, $C_2$. Moreover for any sequence
$(\varepsilon_n) \subset (0, \varepsilon]$ with $\varepsilon_n \to 0$
there exists a subsequence, still denoted by $(\varepsilon_n)$, such that
for each $i \in \{1,\ldots,k\}$ there exist $x^i \in \M^i$ with
$\varepsilon_n x_{\varepsilon_n}^ i \to x^i,$ a constant $w_i \in \R$
and $U_i \in H^1(\R^3, \R)$ a positive least energy solution of
\begin{equation}\label{eq:1.4}
  - \Delta U_i +  m_iU_i - (W*|U_i|^2) U_i = 0, \quad U_i \in H^1(\R^3, \R);
\end{equation}
for which one has
\begin{equation}\label{eq:1.5}
  u_{\varepsilon_n}(x) = \sum_{i=1}^k U_i \left({x-x_{\varepsilon_n}^i} \right) \exp
\left(\iu\left(w_i +
      A(x^i)(x-x_{\varepsilon_n}^i) \right) \right)
  + K_n(x)
\end{equation}
where $K_n \in H_{\varepsilon_n}$ satisfies
$\|K_n\|_{H_{\varepsilon_n}} = o(1)$ as $n \to +\infty$.
\end{theorem}
\vskip8pt

The one and two dimensional cases would require a separate analysis in the construction of
the penalization argument (see e.g.\ \cite{BJT} for a detailed discussion).
The study of the cases of dimensions larger than three is less interesting from the
physical point of view. Moreover, having in mind the soliton dynamics as a possible further development,
in dimensions $N\geq 4$ the time dependent Schr\"odinger
equation with kernels, say, of the type $W(x)=|x|^{2-N}$ does not have global existence in time for all $H^1$
initial data (see e.g.\ \cite[Remark 6.8.2, p.208]{cazenavB}) as well as the
heuristic discussion in the next section).

\subsection{A heuristic remark: multi-bump dynamics}
We could also think of Theorem~\ref{main} as the starting
point in order to rigorously justify a multi-bump soliton dynamics
for the full Schr\"odinger  equation with an external magnetic field
\begin{equation}
\label{FullHartree}
\begin{cases}
{\iu}\ge\partial_t u+\frac{1}{2}\left( \frac{\ge}{\iu} \nabla - A(x) \right)^2 u + V(x) u = \frac{1}{\ge^{2}}\left( W * |u|^2 \right)u
\quad & \text{in $\R^3$}, \\
u(x,0)=u_0(x)\quad & \text{in $\R^3$}.
\end{cases}
\end{equation}
We describe in the following what we expect to hold
(the question is open even for $A=0$, see the discussion by J.\ Fr\"ohlich et al.\ in~\cite{FTY}).
Given $k\geq 1$ positive numbers $g_1,\dots,g_k$, if ${\mathcal E}:H^1(\R^3)\to\R$ is defined as
$$
{\mathcal E}(u)=\frac{1}{2}\int_{\R^3} |\nabla u|^2dx-\frac{1}{4}\int_{\R^6} W(x-y)|u(x)|^2|u(y)|^2dxdy,
$$
let $U_j:\R^3\to\R$ $(j=1,\dots,k)$ be the solutions to the minimum problems
\begin{equation*}
{\mathcal E}(U_j)=\min\{{\mathcal E}(u): u\in H^1(\R^3),\,\|u\|_{L^2}^2=g_j\},
\end{equation*}
which solve the equations
\begin{equation*}
-\frac{1}{2}\Delta U_j+m_jU_j=(W*|U_j|^2)U_j\qquad\text{in $\R^3$},
\end{equation*}
for some $m_j\in\R$. Consider now in~\eqref{FullHartree} an initial datum of the form
\begin{equation*}
u_0(x)=\sum_{j=1}^k U_j\Big(\frac{x-x_0^j}{\eps}\Big)
e^{\frac{{\rm i}}{\eps}[A(x_0^j)\cdot(x-x_0^j)+x\cdot\xi_0^j]},\quad x\in\R^3,
\end{equation*}
where $x_0^j\in\R^3$ and $\xi_0^j\in\R^3$ ($j=1,\dots,k$) are initial position and velocity
for the ODE
\begin{equation}
    \label{newtonODE}
    \begin{cases}
        \dot x_j(t)=\xi_j(t), & \\
        \noalign{\vskip3pt}
        \dot \xi_j(t)=-\nabla V(x_j(t))-
\eps\sum\limits_{i\neq j}^k m_i\nabla W(x_j(t)-x_i(t))-\xi_j(t)\times B(x_j(t)), & \\
        x_j(0)=x_0^j, \quad
        \xi_j(0)=\xi_0^j,\qquad
j=1,\dots,k,
    \end{cases}
\end{equation}
with $B=\nabla \times A$.
The systems can be considered as a mechanical system of
$k$ interacting particles of mass $m_i$ subjected to an external potential as well as
a mutual Newtonian type interaction. Therefore, the conjecture it that, under suitable assumptions, the following
representation formula might hold
\begin{equation}
    \label{conjformula}
u_{\eps}(x,t)=\sum_{j=1}^k U_j\Big(\frac{x-x_j(t)}{\eps}\Big)
e^{\frac{\iu}{\eps}[A(x_j(t))\cdot  (x-x_j(t))+x\cdot  \xi_j(t)+\theta_{\eps}^j(t)]}+\omega_\eps,
\end{equation}
locally in time, for certain phases $\theta_\eps^i:\R^+\to[0,2\pi)$, where $\omega_\eps$ is small
(in a suitable sense) as $\eps\to 0$, provided that the centers $x_0^j$ in the initial data are chosen
sufficiently far from each other. Now, neglecting as $\eps\to 0$ the interaction term ($\eps$-dependent)
$$
\eps\sum\limits_{i\neq j}^k m_i\nabla W(x_j(t)-x_i(t))
$$
in the Newtonian system~\eqref{newtonODE} and taking
$$
x_0^1,\dots,x_0^k\in\R^3: \,\,\nabla V(x_0^j)=0\quad\text{and}\quad
\xi_0^j=0,\,\,\quad\text{for all $j=1,\dots,k$},
$$
then the solution of~\eqref{newtonODE} is
$$
x_j(t)=x_0^j,\,\,\,\xi_j(t)=0,\,\,\quad\text{for all $t\in[0,\infty)$ and $j=1,\dots,k$},
$$
so that the representation formula~\eqref{conjformula} reduces, for $\eps=\eps_n\to 0$,
\begin{equation*}
u_{\eps_n}(x,t)=\sum_{j=1}^k U_j\Big(\frac{x-x^j_0}{\eps_n}\Big)
e^{\frac{\iu}{\eps_n}[A(x^j_0)\cdot  (x-x^j_0)+\theta_{\eps_n}^j(t)]}+\omega_{\eps_n},
\end{equation*}
namely to formula~\eqref{eq:1.5}
up to a change in the phase terms and up to replacing $x$ with $\eps_n x$ and $x_0^j$
with $\eps_n x_{\eps_n}^j$ for all $j=1,\dots,k$.

\vskip15pt
\begin{center}\textbf{Plan of the paper.}\end{center}
In Section 2 we obtain several results about the structure of the solutions of the limiting problem~\eqref{eq:1.4}.
In particular, we study the compactness of the set of real ground states solutions and we achieve
a result about the orbital stability property of these solutions for the Pekar-Choquard equation.
In Section 3 we perform the penalization scheme. In particular we obtain various energy
estimates in the semiclassical regime $\eps\to 0$ and we get a Palais-Smale condition for the
penalized functional which allows to find suitable critical points inside the concentration set.
Finally we conclude the proof of Theorem~\ref{main}.

\vskip15pt
\begin{center}\textbf{Main notations.}\end{center}
\begin{enumerate}
\item $\iu$ is the imaginary unit.
\item The complex conjugate of any number $z\in\C$ is denoted by $\bar z$.
\item The real part of a number $z\in\C$ is denoted by $\Re z$.
\item The imaginary part of a number $z\in\C$ is denoted by $\Im z$.
\item The symbol $\R^+$ (resp.\ $\R^-$) means the positive real line $[0,\infty)$ (resp.\ $(-\infty,0]$).
\item The ordinary inner product between two vectors $a,b\in\R^3$ is denoted by $\langle a \mid b \rangle$.
\item The standard $L^p$ norm of a function $u$ is denoted by $\|u\|_{L^p}$.
\item The standard $L^\infty$ norm of a function $u$ is denoted by $\|u\|_{L^\infty}$.
\item The symbol $\Delta$ means $D^2_{x_1}+D^2_{x_2}+D^2_{x_3}$.
\item The convolution $u*v$ means $(u*v)(x)=\int u(x-y)v(y)dy$.
\end{enumerate}
\smallskip

\section{Properties of the set of ground states}

For any positive real number $a$, the limiting equation for the
Hartree problem~\eqref{eq:Hartree} is
\begin{equation}
  \label{eq:limiting}
  -\Delta u +au=\left( W * |u|^2 \right) u\qquad\text{in $\R^3$}.
\end{equation}

\subsection{A Pohozaev type identity}

We now give the statement of a useful identity satisfied by solutions
to problem~\eqref{eq:limiting}.

\begin{lemma}
\label{pohoz}
Let $u \in H^1(\R^3,\mathbb{C})$ be a solution to~\eqref{eq:limiting}. Then
\begin{equation}
\label{eq:poho}
\frac{1}{2}\int_{\R^3} |\nabla u|^2 \, dx + \frac{3}{2}a \int_{\R^3} |u|^2  dx
= \frac{5}{4} \int_{\R^3 \times \R^3} W(x-y) |u(x)|^2 |u(y)|^2  dx dy.
\end{equation}
\end{lemma}
\begin{proof}
The proof is straightforward, and we include it just for
the sake of completeness. It is enough to prove it for smooth functions, using then a standard density argument.
We multiply equation~\eqref{eq:limiting}
by $\langle x \mid \overline{\nabla u} \rangle$. Notice that
\begin{align}
\label{eq:a}
    \Re \left(\Delta u \langle x \mid \overline{\nabla u} \rangle \right) &= \operatorname{div} \Big( \Re \left(\langle x \mid \overline{\nabla u} \rangle \nabla u \right) - \frac{1}{2} |\nabla u|^2 x \Big) + \frac{1}{2}|\nabla u|^2,  \\
    \label{eq:b}
    \Re \left(-a u \langle x \mid \overline{\nabla u} \rangle \right)&= -a \operatorname{div} \Big( \frac{1}{2} |u|^2 x \Big) + \frac{3}{2} a |u|^2, \\\
    \varphi(x) \Re \left(u \langle x \mid \overline{\nabla u} \rangle \right) &= \operatorname{div}\Big( \frac{1}{2} |u|^2 \varphi(x) x \Big) - \frac{1}{2} |u|^2 \operatorname{div} \left( \varphi(x)x \right), \label{eq:c}
\end{align}
where $\varphi(x)=\int_{\R^3} W(x-y) |u(y)|^2 dy$. We can easily obtain that
\begin{align*}
  \operatorname{div} ( \varphi(x)x ) &= \sum_{i=1}^3 \frac{\partial}{\partial x_i} \Big( x_i \int_{\R^3} W(x-y) |u(y)|^2\, dy \Big) \\
  &= N \int_{\R^3} W(x-y) |u(y)|^2 \, dy +\int_{\R^3} \langle \nabla W(x-y) \mid x \rangle |u(y)|^2  dy.
  \end{align*}
Summing up~\eqref{eq:a},~\eqref{eq:b} and~\eqref{eq:c} and integrating by parts, we reach the identity
\begin{multline}\label{eq:3}
\frac{1}{2} \int_{\R^3} |\nabla u|^2 \, dx + \frac{3}{2} a \int_{\R^3} u^2 \, dx - \frac{3}{2} \int_{\R^3 \times \R^3} W(x-y) |u(x)|^2 |u(y)|^2 \, dx \, dy  \\ {} - \frac{1}{2} \int_{\R^3 \times \R^3} \langle \nabla W(x-y) \mid x \rangle |u(x)|^2 |u(y)|^2  dx  dy =0.
\end{multline}
By exchanging $x$ with $y$, we find that
\begin{multline*}
\int_{\R^3 \times \R^3} \langle \nabla W(x-y) \mid x \rangle |u(x)|^2 |u(y)|^2 \, dx \, dy = \\
- \int_{\R^3 \times \R^3} \langle \nabla W(x-y) \mid y \rangle |u(x)|^2 |u(y)|^2 \, dx \, dy.
\end{multline*}
Therefore,
\begin{multline*}
\int_{\R^3 \times \R^3} \langle \nabla W(x-y) \mid x \rangle |u(x)|^2 |u(y)|^2 \, dx \, dy \\
= \frac{1}{2} \int_{\R^3 \times \R^3} \langle \nabla W(x-y) \mid x-y \rangle |u(x)|^2 |u(y)|^2 \, dx \, dy \\
= -\frac{1}{2} \int_{\R^3 \times \R^3} W(x-y)  |u(x)|^2 |u(y)|^2 \, dx \, dy
\end{multline*}
via Euler's identity for homogeneous functions. Plugging this into~\eqref{eq:3} yields~\eqref{eq:poho}.
\end{proof}

\subsection{Orbital stability property}

In this section, we consider the Schr\"odinger equation
\begin{equation}
    \label{evolSE}
\begin{cases}
\iu\frac{\partial u}{\partial t}+\Delta u+(W * |u|)^2u=0  & \text{in $\R^3\times(0,\infty)$}, \\
u(x,0)=u_0(x) & \text{in $\R^3$},
\end{cases}
\end{equation}
under assumption, coming from ${\bf (W)}$,  that
\begin{equation}
    \label{carab}
\frac{C_1}{|x|} \leq W(x) \leq \frac{C_2}{|x|},
\end{equation}
for positive constants $C_1,C_2$ (cf.~\eqref{growth}).
This equation is also known as Pekar-Choquard equation (see e.g.\ \cite{CL-cmp,liebchoq,Lionchoq}).
Consider the functionals
\begin{equation*}
{\mathcal E}(u) = \frac{1}{2} \|\nabla u\|^2_{L^2}- \frac{1}{4} \D(u),\quad
J(u) = \frac{1}{2} \|\nabla u\|^2_{L^2}+ \frac{a}{2}\|u\|^2_{L^2}- \frac{1}{4} \D(u),
\end{equation*}
where
\begin{equation}
\label{defDD}
\D(u)= \int_{\R^6}W(x-y)|u(x)|^2 |u(y)|^2\, dx dy,
\end{equation}
and let us set
\begin{align*}
{\mathcal M}& =\big\{u\in H^1(\R^3,\C):\,\|u\|_{L^2}^2=\rho\big\}, \\
\noalign{\vskip2pt}
{\mathcal N}& =\big\{u\in H^1(\R^3,\C):\,\text{$u\neq 0$ and $J'(u)(u)=0$}\},
\end{align*}
for some positive number $\rho>0$.
\vskip3pt

\begin{definition}
    \label{defgstates}
  We denote by ${\mathcal G}$ the set of ground state solutions
  of~\eqref{eq:limiting}, that is solutions to the minimization problem
  \begin{equation}
    \label{MinPb1}
    \Lambda=\min_{u\in{\mathcal N}} J(u).
  \end{equation}
\end{definition}

\begin{remark}\rm
    \label{nonEmN}
    By Corollary~\ref{addprop} (see also the correspondence between
    critical points in the proof of Lemma~\ref{equiv-min}), the minimization problem in Definition~\ref{defgstates}
is equivalent to a constrained minimization problem on a sphere of $L^2$ with a
 suitable radius $\rho$. For the latter problem one can find a solution
by following the arguments of~\cite[Section 1]{CL-cmp}, as minimizing sequences
converge strongly in $H^1(\R^3)$. In particular, ${\mathcal G}\not=\emptyset$.
\end{remark}

In Lemma~\ref{equiv-min} we will prove that a ground state
solution of~\eqref{eq:limiting} can be obtained as scaling of a
 solution to the minimization problem
  \begin{equation}
    \label{MinPb2}
    \Lambda=\min_{u\in{\mathcal M}} {\mathcal E}(u),
  \end{equation}
which is a quite useful characterization for the stability issue.
We now recall two global existence results for problem~\eqref{evolSE}
(see e.g.\ \cite[Corollary 6.1.2, p.164]{cazenavB}). We remark that
\eqref{carab} holds.

\begin{proposition}
    \label{glob1}
  Let $u_0\in H^1(\R^3)$. Then
  problem~\eqref{evolSE} admits a unique global solution $u\in
  C^1([0,\infty),H^1(\R^3,\C))$. Moreover, the charge and the energy
  are conserved in time, namely
  \begin{equation}
    \label{chargeenergy}
    \|u(t)\|_{L^2}=\|u_0\|_{L^2},\qquad {\mathcal E}(u(t))={\mathcal E}(u_0),
  \end{equation}
  for all $t\in[0,\infty)$.
\end{proposition}

\begin{definition}
  The set ${\mathcal G}$ of ground state solutions
  of~\eqref{eq:limiting} is said to be orbitally stable for the
  Pekar-Choquard equation~\eqref{evolSE} if for every $\eps>0$ there
  exists $\delta>0$ such that
\begin{equation*}
  \text{$\forall u_0\in H^1(\R^3,\C)$:\quad $\inf_{\phi\in{\mathcal G}}\|u_0-\phi\|_{H^1}<\delta$\quad implies that \quad
    $\sup_{t\geq 0}\inf_{\psi\in {\mathcal G}}\|u(t,\cdot)-\psi\|_{H^1}<\eps$},
\end{equation*}
where $u(t,\cdot)$ is the solution of~\eqref{evolSE} corresponding to
the initial datum $u_0$.
\end{definition}

Roughly speaking, the ground states are orbitally stable if any
orbit starting from an initial datum $u_0$ close to ${\mathcal G}$
remains close to ${\mathcal G}$, uniformly in time.

In the classical orbital stability of Cazenave and Lions (see e.g.\ \cite{CL-cmp})
the ground states set ${\mathcal G}$ is meant as the set of minima of the functional
${\mathcal E}$ constrained to a sphere of $L^2(\R^3)$. In this section we just aim to show that
orbital stability holds with respect to ${\mathcal G}$ as defined in Definition~\ref{defgstates}.

\vskip4pt

Consider the following sets:
\begin{align*}
K_{\mathcal N}& =\{m\in\R:\text{there is $w\in {\mathcal N}$ with $J'(w)=0$ and $J(w)=m$}\}, \\
K_{\mathcal M}& =\{c\in\R^-:\text{there is $u\in {\mathcal M}$ with ${\mathcal E}'|_{{\mathcal M}}(u)=0$
and ${\mathcal E}(u)=c$}\}.
\end{align*}
In the next result we establish the equivalence between
minimization problems~\eqref{MinPb1} and~\eqref{MinPb2}, namely that a suitable scaling of a
solution of the first problem corresponds to a solution of the second
problem with a mapping between the critical values.

\begin{lemma}
  \label{equiv-min}
  The following minimization problems are equivalent
  \begin{equation}
    \label{defdei2valori}
    \Lambda=\min_{u\in{\mathcal M}} {\mathcal E}(u),\qquad
    \Gamma=\min_{u\in{\mathcal N}} J(u),
  \end{equation}
for $\Lambda<0$ and $\Lambda=\Psi(\Gamma)$, where $\Psi:K_{\mathcal N}\to K_{\mathcal M}$ is defined by
$$
\Psi(m)=
-\frac{1}{2}\Big(\frac{3}{a\rho}\Big)^{-3}
  m^{-2},\quad m\in K_{\mathcal N}.
$$
\end{lemma}
\begin{proof}
  Observe that if $u\in{\mathcal M}$ is a critical point of ${\mathcal
    E}|_{{\mathcal M}}$ with ${\mathcal E}(u)=c<0$, then there exists
  a Lagrange multiplier $\gamma\in\R$ such that ${\mathcal
    E}'(u)(u)=-\gamma\rho$, that is $\|\nabla
  u\|^2_{L^2}-\D(u)=-\gamma\rho$. By combining this identity with
  $\D(u)=2\|\nabla u\|^2_{L^2}-4c$, we obtain $-\|\nabla
  u\|^2_{L^2}+4c=-\gamma\rho$, which implies that $\gamma>0$.  The equation satisfied by $u$ is
  \begin{equation*}
    -\Delta u +\gamma u= \left( W * |u|^2 \right) u \qquad\text{in $\R^3$}.
  \end{equation*}
  After trivial computations one shows that the scaling
  \begin{equation}
    \label{rescalingfurb}
    w(x)=T^\lambda u(x):=\lambda^2u(\lambda x),\qquad \lambda:=\sqrt{\frac{a}{\gamma}}
  \end{equation}
  is a solution of equation~\eqref{eq:limiting}. On the contrary, if
  $w$ is a nontrivial critical point of $J$, then choosing
  \begin{equation}
    \label{suitabchoi}
    \lambda=\rho^{-1}\|w\|_{L^2}^{2},
  \end{equation}
  the function $u=T^{1/\lambda}w$ belongs to ${\mathcal M}$ and it is
  a critical point of ${\mathcal E}_{|\mathcal M}$. Now, Let $m$ be
  the value of the free functional $J$ on $w$, $m=J(w)$. Then
\begin{align}
  \label{enbalanc}
  m& =\frac{1}{2} \|\nabla w\|^2_{L^2}+ \frac{a}{2}\|w\|^2_{L^2}- \frac{1}{4} \D(w) \\
  &=\frac{1}{2}\lambda^{3}\|\nabla u\|^2_{L^2}+\frac{a}{2}\lambda\|u\|^2_{L^2}
-\frac{1}{4}\lambda^{3}\D(u) \notag \\
  & =\lambda^{3}{\mathcal E}(u)+\frac{a}{2}\lambda\rho   \notag\\
  & =c\lambda^{3}+\frac{a}{2}\lambda\rho.  \notag
\end{align}
Observe that, since of course $\D(w)=\|\nabla
w\|^2_{L^2}+a\|w\|^2_{L^2}$ and $w$ satisfies the Pohozaev
identity~\eqref{eq:poho}, we have the system
\begin{align*}
  & \frac{1}{2}\|\nabla w\|^2_{L^2} + \frac{3}{2}a \|w\|^2_{L^2}= \frac{5}{4} \big( \|\nabla w\|^2_{L^2}+a\|w\|^2_{L^2} \big),   \\
  & \frac{1}{4} \|\nabla w\|^2_{L^2}+ \frac{a}{4}\|w\|^2_{L^2}=m,
\end{align*}
namely
\begin{align*}
  &  3\|\nabla w\|^2_{L^2} - a \|w\|^2_{L^2}=0,  \\
   & \|\nabla w\|^2_{L^2}+ a\|w\|^2_{L^2}=4m,
\end{align*}
which, finally, entails
\begin{equation}
  \label{finalconcl}
  \|\nabla w\|^2_{L^2}=m,\qquad
  \|w\|^2_{L^2}=\frac{3}{a}m.
\end{equation}
As a consequence a simple rescaling yields the value of $\lambda$,
that is
\[
\rho\lambda=\|w\|^2_{L^2}=\frac{3}{a}m.
\]
Replacing this value of $\lambda$ back into
formula~\eqref{enbalanc}, one obtains
\[
m=c\Big(\frac{3m}{a\rho}\Big)^{3}+\frac{3}{2}m,
\]
namely
\[
-\frac{1}{2}m=c\Big(\frac{3m}{a\rho}\Big)^{3}.
\]
In conclusion, we get
\begin{equation}
  \label{formulaeneergies}
 c=\Psi(m)\overset{\mathrm{def}}{=}-\frac{1}{2}
\Big(\frac{3}{a\rho}\Big)^{-3}m^{-2},
\end{equation}
where the function $\Psi:\R^+\to\R^-$ is injective.
Of course
formally $m=\Psi^{-1}(c)$, where
\[
\Psi^{-1}(c)=\sqrt{\mathstrut -\frac{1}{2c}\Big( \frac{3}{a\rho} \Big)^{-3}},
\]
which is injective.
In order to
prove that $\Psi^{-1}$ is surjective, let $m$ be a free critical value for
$J$, namely $m=J(w)$, with $w$ solution of
equation~\eqref{eq:limiting}. Then, if we consider $u=T^{1/\lambda}
w(x)=\lambda^{-2}w(\lambda^{-1} x)$ with $\lambda$ given
by~\eqref{suitabchoi}, it follows that $u\in {\mathcal M}$ is a
critical point of ${\mathcal E}|_{\mathcal M}$ with lagrange
multiplier $\gamma=a\lambda^{-2}$. By using
\[
\lambda=\left(\frac{a\rho}{3m}\right)^{-1},
\]
in light of~\eqref{finalconcl} we have
\begin{align*}
  4c=\|\nabla u\|^2_{L^2}-\gamma\rho &=\lambda^{-3}\|\nabla w\|^2_{L^2}-a\rho\lambda^{-2} \\
&  = \Big(\frac{a\rho}{3m}\Big)^{3}m-a\rho\Big(\frac{3m}{a\rho}\Big)^{-2},
\end{align*}
which yields $m=\Psi^{-1}(c)$, after a few computations. We are now
ready to prove the assertion.  Notice that by
formula~\eqref{formulaeneergies} we have
\begin{align*}
  \Lambda & =\min_{u\in{\mathcal M}} {\mathcal E}(u)=\min_{u\in{\mathcal M}} c_u \\
  & =\min_{v\in{\mathcal N}} \Psi(m_v)=-\max_{v\in{\mathcal N}}
  \frac{1}{2}\Big(\frac{3}{a\rho}\Big)^{-3}
  m_v^{-2} \\
  &
  =-\frac{1}{2}\Big(\frac{3}{a\rho}\Big)^{-3}
  \Big( \min_{v\in{\mathcal N}} m_v \Big)^{-2} \\
  &=-\frac{1}{2}\Big(\frac{3}{a\rho}\Big)^{-3}
  \Gamma^{-2}=\Psi(\Gamma).
\end{align*}
If $\hat u\in{\mathcal M}$ is a minimizer for $\Lambda$, that is
$\Lambda={\mathcal E}(\hat u)=\min_{{\mathcal M}} {\mathcal E}$, the
function $\hat w=T^\lambda \hat u$ is a critical point of $J$ with
$J(\hat w)=\Psi^{-1}(\Lambda)=\Gamma$, so that $w$ is a minimizer for
$\Gamma$, that is $J(w)=\min_{{\mathcal N}}J$. This concludes the
proof.
\end{proof}

\begin{corollary}
    \label{addprop}
Any ground state solution $u$ to equation~\eqref{eq:limiting} satisfies
\begin{equation}
    \label{L^2const}
    \|u\|^2_{L^2}=\rho,\qquad\rho=\frac{3\Gamma}{a},
\end{equation}
where $\Gamma$ is defined in~\eqref{defdei2valori}. Moreover, for
this precise value of the radius $\rho$, we have
 \begin{equation}
    \label{ugvalori}
    \min_{u\in{\mathcal M}} J(u)=\min_{u\in{\mathcal N}} J(u),
  \end{equation}
where ${\mathcal M}={\mathcal M}_\rho$.
\end{corollary}
\begin{proof}
The first conclusion is an immediate consequence of the previous proof. Let us now
prove that the second conclusion holds, with $\rho$ as in~\eqref{L^2const}.
We have
\begin{align*}
    \min_{u\in{\mathcal M}} J(u) & =\min_{u\in{\mathcal M}} {\mathcal E}(u)+\frac{a}{2}\|u\|^2_{L^2}
    =\Lambda+\frac{a\rho}{2} \\
    & =-\frac{1}{2}\Big(\frac{3}{a\rho}\Big)^{-3}\Gamma^{-2}+\frac{a\rho}{2}
   \\
&=\Gamma=\min_{u\in{\mathcal N}} J(u),
\end{align*}
by the definition of $\rho$.
\end{proof}

The following is the main result of the section.

\begin{theorem}
    \label{orbitstab}
  Then the set ${\mathcal G}$ of the ground state solutions
  to~\eqref{eq:limiting} is orbitally stable for~\eqref{evolSE}.
\end{theorem}
\begin{proof}
  Assume by contradiction that the assertion is false. Then we can
  find $\eps_0>0$, a sequence of times $(t_n)\subset(0,\infty)$ and of
  initial data $(u_0^n)\subset H^1(\R^3,\C)$ such that
  \begin{equation}
    \label{contradarg}
    \text{$\lim_{n\to\infty}\inf_{\phi\in{\mathcal G}}\|u_0^n-\phi\|_{H^1}=0$\quad  and \quad
      $\inf_{\psi\in {\mathcal G}}\|u_n(t_n,\cdot)-\psi\|_{H^1}\geq\eps_0$},
  \end{equation}
  where $u_n(t,\cdot)$ is the solution of~\eqref{evolSE} corresponding
  to the initial datum $u_0^n$. Taking into account~\eqref{L^2const}
and~\eqref{ugvalori} of Corollary~\ref{addprop}, for any $\phi\in{\mathcal G}$,
we have
\[
\|\phi\|^2_{L^2}=\rho_0,\qquad
J(\phi)=\min_{u\in{\mathcal M}_{\rho_0}} J(u),
\qquad
\quad\rho_0=\frac{3\Gamma}{a}.
\]
Therefore, considering the sequence $\Upsilon_n(x):=u_n(t_n,x)$, which is
bounded in $H^1(\R^3,\C)$, and recalling
the conservation of charge, as $n\to\infty$, from~\eqref{contradarg} it follows that
\begin{equation*}
  \|\Upsilon_n\|_{L^2}^2 =\|u_n(t_n,\cdot)\|_{L^2}^2=\|u^n_0\|_{L^2}^2=\rho_0+o(1).
  \end{equation*}
Hence, there exists a sequence $(\omega_n)\subset\R^+$ with $\omega_n\to 1$ as $n\to\infty$
such that
\begin{equation}
  \label{1contr}
  \|\omega_n\Upsilon_n\|_{L^2}^2 =\rho_0,\quad
\text{for all $n\geq 1$}.
 \end{equation}
Moreover, by the conservation of energy~\eqref{chargeenergy} and
the continuity of ${\mathcal E}$, as $n\to\infty$,
\begin{align}
  \label{2contr}
J(\omega_n\Upsilon_n) & ={\mathcal E}(\omega_n\Upsilon_n)+\frac{a}{2}\|\omega_n\Upsilon_n\|_{L^2}^2
={\mathcal E}(\Upsilon_n)+\frac{a}{2}\|\Upsilon_n\|_{L^2}^2+o(1) \\
&={\mathcal E}(u_n(t_n,\cdot))+\frac{a}{2}\|u_0^n\|_{L^2}^2+o(1)
={\mathcal E}(u^n_0)+\frac{a}{2}\|u_0^n\|_{L^2}^2+o(1) \notag\\
&=J(u_0^n)+o(1)=\min_{u\in{\mathcal M}_{\rho_0}} J(u)+o(1). \notag
\end{align}
Combining~\eqref{1contr}-\eqref{2contr}, it follows that $(\omega_n\Upsilon_n)\subset H^1(\R^3,\C)$ is a
minimizing sequence for the functional $J$ (and also for ${\mathcal E}$) over ${\mathcal M}_{\rho_0}$. By
taking into account the homogeneity property of $W$,
following the arguments of~\cite[Section 1]{CL-cmp}, we deduce that, up to
a subsequence, $(\omega_n\Upsilon_n)$ converges to some function $\Upsilon_0$,
which thus belongs to the set ${\mathcal G}$, since
by~\eqref{1contr}-\eqref{2contr} and equality \eqref{ugvalori}
$$
J(\Upsilon_0)=\min_{u\in{\mathcal M}_{\rho_0}} J(u)=\min_{u\in{\mathcal N}} J(u)
$$
Evidently, this is a contradiction with~\eqref{contradarg}, as we would have
\[
\eps_0\leq \lim_{n\to\infty}\inf_{\psi\in {\mathcal
    G}}\|\Upsilon_n-\psi\|_{H^1}\leq \lim_{n\to\infty}
\|\Upsilon_n-\Upsilon_0\|_{H^1}=0.
\]
This concludes the proof.
\end{proof}

In the particular case $W(x)=|x|^{-1}$, due to the uniqueness of ground states up to translations
and phase changes (cf.\ \cite{moroz}), Theorem~\ref{orbitstab} strengthens as follows.

\begin{corollary}
    Assume that $w$ is the unique real ground state of
    $$
    -\Delta w+aw=(|x|^{-1}*w^2)w,\quad\text{in $\R^3$}.
    $$
    Then for all $\eps>0$ there exists $\delta>0$ such that
\begin{equation*}
u_0\in H^1(\R^3,\C)\quad\text{and}\quad \inf_{\overset{y\in\R^3}{\theta\in[0,2\pi)}}
\|u_0-e^{\iu\theta}w(\cdot-y)\|_{H^1}<\delta
\end{equation*}
implies that
\begin{equation*}
   \sup_{t\geq 0}\inf_{\overset{y\in\R^3}{\theta\in[0,2\pi)}}\|u(t,\cdot)-e^{\iu\theta}w(\cdot-y)\|_{H^1}<\eps.
\end{equation*}
\end{corollary}

\subsection{Structure of least energy solutions}

We can now state the following

\begin{lemma}
  \label{RemRaprGS}
  Any complex ground state solution $u$ to~\eqref{eq:limiting} has the form
  \begin{equation*}
    u(x)=e^{\iu\theta}|u(x)|,\,\quad\text{for some $\theta\in[0,2\pi)$}.
  \end{equation*}
\end{lemma}
\begin{proof}
  In view of Lemma~\ref{equiv-min} (see also Corollary~\ref{addprop}), searching for ground state
  solutions of~\eqref{eq:limiting} is equivalent to consider the
  constrained minimization problem $\min_{u\in{\mathcal M}_\rho} {\mathcal
    E}(u)$ for a suitable value of $\rho>0$. Then the proof is quite standard; we include
a proof here for the sake of selfcontainedness. Consider
\begin{align*}
  \sigma_\C&=\inf\big\{\E(u): u\in H^1(\R^3,\C),\,\|u\|_{L^2}^2=\rho \big\},\\
  \sigma_\R&=\inf\big\{\E(u): u\in H^1(\R^3,\R), \|u\|_{L^2}^2=\rho \big\}.
\end{align*}
It holds $\sigma_\C=\sigma_\R$. Indeed, trivially one has
$\sigma_\C\leq \sigma_\R$. Moreover, if $u\in H^1(\R^3,\C)$, due to
the well-known inequality $|\nabla |u(x)||\leq |\nabla u(x)|$ for
a.e.\ $x\in\R^3$, it holds
\[
\int|\nabla |u(x)||^2dx\leq \int|\nabla u(x)|^2dx,
\]
so that $\E(|u|)\leq \E(u)$. In particular $\sigma_\R\leq\sigma_\C$,
yielding $\sigma_\C=\sigma_\R$.  Let now $u$ be a solution to
$\sigma_\C$ and assume by contradiction that $\mu(\{x\in\R^3:|\nabla |u|(x)|<|\nabla u(x)|\})>0,$ where
$\mu$ denotes the Lebesgue measure in $\R^3$.  Then
$\||u|\|_{L^2}=\|u\|_{L^2}=1$, and
\begin{equation*}
  \sigma_\R \leq \frac{1}{2}\int |\nabla |u||^2dx-\frac{1}{4}\D(|u|)
  <\frac{1}{2}\int |\nabla u|^2dx-\frac{1}{4}\D(u)=\sigma_\C,
\end{equation*}
which is a contradiction, being $\sigma_\C=\sigma_\R$. Hence, we have
$|\nabla |u(x)||=|\nabla u(x)|$ for a.e.\ $x\in\R^3$. This is true if
and only if $\re\, u\nabla (\im\, u)=\im\, u\nabla(\re\, u)$.  In
turn, if this last condition holds, we get
\[
{\bar u}\nabla u=\re\, u\nabla (\re\, u)+ \im\, u\nabla (\im\,
u),\quad \text{a.e.\ in $\R^3$},
\]
which implies that $\re\,(i\bar u(x)\nabla u(x))=0$ a.e.\ in
$\R^3$. From the last identity one finds $\theta\in [0,2\pi)$ such
that $u=e^{\iu\theta}|u|$, concluding the proof.
\end{proof}

We then get the following result about least-energy levels for the limiting problem~\eqref{eq:limiting}.
\begin{corollary}\label{levels}
Consider the two problems
\begin{eqnarray}
-\Delta u +au=  (W * |u|^2)  u,\quad &&u \in H^1(\mathbb{R}^3,\mathbb{R}),  \label{eq:2.17}\\
-\Delta u +au=  (W * |u|^2)  u,\quad &&u \in H^1(\mathbb{R}^3,\mathbb{C}), \label{eq:2.18}.
\end{eqnarray}
Let $E_a$ and $E_a^c$ denote their least-energy levels. Then
\begin{equation}
    \label{eq:uguali}
 E_a= E_a^c.
 \end{equation}
Moreover any least energy solution of \eqref{eq:2.17} has the form
$e^{i \tau} U$ where $U$ is a positive least energy solution of
\eqref{eq:2.18} and $\tau \in \R$.
\end{corollary}

\subsection{Compactness of the ground states set}

In light of assumption~\textbf{(W)}, there exist two positive
constants $C_1,C_2$ such that
\begin{equation}
  \label{growth}
  \frac{C_1}{|x|}\leq W(x)\leq \frac{C_2}{|x|},\quad\text{for all $x\in \R^3\setminus\{0\}$}.
\end{equation}

\vskip2pt
\noindent
We recall two Hardy-Littlewood-Sobolev type inequality (see e.g.\
\cite[Theorem 1, pag. 119]{stein} and \cite[Theorem 4.3, p.98]{ll}) in $\R^3$:
\begin{align}
  \label{Hardyin}
&  \forall u\in L^{\frac{6q}{3+2q}}(\R^3):\quad \left\| |x|^{-1} * u^2  \right\|_{L^q} \leq C \|u\|^2_{L^{\frac{6q}{3+2q}}}, \\
  \label{HLS2}
 & \forall u\in L^{\frac{12}{5}}(\R^3):\quad
  \int_{\R^{6}}|x-y|^{-1}u^2(y)u^2(x)dydx\leq C\|u\|_{L^{\frac{12}{5}}}^4.
\end{align}

\vspace{2pt}

We have the following
\begin{lemma}
  \label{finitEn}
  There exists a positive constant $C$ such that
  \begin{equation*}
    \forall u\in H^1(\R^3):\quad
    \D(u)\leq C\|u\|^{3}_{L^2}\,\|u\|_{H^1}.
  \end{equation*}
\end{lemma}
\begin{proof}
  By combining~\eqref{growth},~\eqref{HLS2} and the
  Gagliardo-Nirenberg inequality, we obtain
  \begin{eqnarray*}
    \D(u)\leq C_2 \int_{\R^{6}}|x-y|^{-1}u^2(x)u^2(y)\, dxdy
    \le C\|u\|^4_{L^{\frac{12}{5}}}
    \le C\|u\|^{3}_{L^2}\,\|u\|_{H^1},
  \end{eqnarray*}
  which proves the assertion.
\end{proof}

More generally, we recall the following facts from \cite[Section 2]{macri}.

\begin{lemma} \label{lemma:2.12}
Assume that $K \in L^s(\R^3) + L^\infty(\R^3)$ for some $s \geq 3/2$.
Then there exists a constant $C>0$ such that, for any $u,v \in H^1(\R^3)$,
\begin{equation}\label{stima:infty}
\|K * (uv) \|_{L^\infty} \leq C \|u\|_{H^1} \|v\|_{H^1}.
\end{equation}
Moreover, assume that $K \in L^{3-\eps}(\R^3) + L^{3+\eps}(\R^3)$ for some $\eps>0$ small.
Then
\begin{equation}
    \label{tozeroconv}
\lim_{|x|\to\infty} (K*u^2)(x)=0,
\end{equation}
for any $u\in H^1(\R^3)$.
\end{lemma}

\begin{proof}
Write $K=K_1+K_2$ with $K_1 \in L^\infty(\R^3)$ and $K_2 \in L^s (\R^3)$, $s \geq 3/2$.
Then, by Sobolev's embedding theorems, $H^1(\R^3) \subset
L^{2s/(s-1)}(\R^3)$ and hence, for some $C>0$,
\begin{align*}
\|K*(uv)\|_{L^\infty} &= \|(K_1+K_2)*(uv)\|_{L^\infty} \\
&\leq \|K_1\|_{L^\infty} \|u\|_{L^2} \|v\|_{L^2}+
\|K_2\|_{L^s} \|uv\|_{L^{s/(s-1)}} \\
&\leq \|K_1\|_{L^\infty} \|u\|_{H^1} \|v\|_{H^1}+ \|K_2\|_{L^s} \|u\|_{L^{2s/(s-1)}} \|v\|_{L^{2s/(s-1)}}
\leq C  \|u\|_{H^1} \|v\|_{H^1}.
\end{align*}
Concerning the second part of the statement,
$|K*u^2(x)|\leq |K_1*u^2(x)|+|K_2*u^2(x)|$ for every $x\in\R^3$ and,
for $\eps>0$ small, by Sobolev embedding $u^2\in L^{(3-\eps)'}(\R^3)\cap L^{(3+\eps)'}(\R^3)$,
where $(3-\eps)'$ and $(3+\eps)'$ are the conjugate exponents of
$3-\eps$ and $3+\eps$ respectively. Then the assertion follows directly from
\cite[Lemma 2.20, p.70]{ll}.
\end{proof}
Let ${\mathcal S}_a$ denote the set of (complex) least energy
solutions $u$ to equation~\eqref{eq:limiting} such that
$$
|u(0)|=\max_{x \in\mathbb{R}^3} |u(x)|.
$$
By Lemma~\ref{RemRaprGS}, up to a
constant phase change, we can assume that $u$ is real valued.
Moreover ${\mathcal S}_a\not=\emptyset$, see Remark~\ref{nonEmN}.

\begin{proposition} \label{prop:compact}
For any $a>0$ the set ${\mathcal S}_a$ is compact in $H^1(\R^3,\R)$
and there exist positive constants $C,\sigma$ such that $u(x)\leq
C\exp(-\sigma |x|)$ for any $x\in\R^3$ and all $u\in{\mathcal S}_a$.
\end{proposition}

\begin{proof}
If $L_a:H^1(\R^3) \to \R$ denotes the functional associated with~\eqref{eq:limiting},
$$
L_a (u) =L(u)=\frac{1}{2} \int_{\R^3} |\nabla u|^2dx + \frac{a}{2} \int_{\R^3}u^2dx - \frac{1}{4} \D(u),
$$
since $\D(u)=\|\nabla u\|^2_{L^2} + a\|u\|^2_{L^2}$ for all $u\in \mathcal{S}_a$, we have
\begin{equation*}
m_a=L_a(u) = \frac{1}{4}\big( \|\nabla u\|^2_{L^2} + a \|u\|^2_{L^2} \big)  ,
\end{equation*}
where $m_a=\min\{L_a(u):\text{$u\neq 0$ solves
\eqref{eq:limiting}}\}$. Hence, it follows that the set $\mathcal{S}_a$ is bounded in $H^1(\R^3)$. Moreover, ${\mathcal S}_a$ is
also bounded in $L^\infty(\R^3)$ and $u(x)\to 0$ as $|x|\to\infty$
for any $u\in {\mathcal S}_a$. Indeed, from the
Hardy-Littlewood-Sobolev inequality~\eqref{Hardyin}, for any $q\geq 3$
\begin{equation} \label{eq:tutti_q}
\|W*u^2\|_{L^{q}} \leq C \| |x|^{-1} * u^2\|_{L^{q}} \leq C\|u\|^2_{L^{6q/(3+2q)}}\leq C \|u\|^2_{L^{6}}\leq C,
\end{equation}
so that $|x|^{-1} * u^2\in L^q(\R^3)$, for any $q\geq 3$.
Setting
$$
f(x)=(W* u^2)(x)u(x)-au(x),
$$
for all $m$ with $2\leq m<6$, by H\"older and Sobolev inequalities, we get
\[
\|f\|_{L^m}\leq C\|u\|_{L^m}+C \|(|x|^{-1} * u^2)u\|_{L^m}\leq
C\|u\|_{L^m}+ C\| |x|^{-1} *
u^2\|_{L^{\frac{6m}{6-m}}}\|u\|_{L^{6}}\leq C.
\]
By virtue of~\eqref{tozeroconv} of Lemma~\ref{lemma:2.12} we have
$(W*u^2)(x)\to 0$ as $|x|\to\infty$, and thus standard arguments
show that $u$ is exponentially decaying to zero at infinity (see
also the argument just before~\eqref{toexpdec}) which readily
implies  $f(x) \in L^1(\R^3)$. From equation~\eqref{eq:limiting},
namely $-\Delta u=f$, by Calderon-Zygmund estimate
(see~\cite[Theorem 9.9, Corollary 9.10 and lines just before
it]{GT}; note also that $f\in L^1(\R^3)$ so that by means of
\cite[Lemma A.5, p.550]{bebrcr} it holds $u=G*f$, where $G$ is the
fundamental solution of $-\Delta$ on $\R^N$), it follows that $u\in
W^{2,m}(\R^3)$, for every $2\leq m<6$, with $\|u\|_{W^{2,m}}$
uniformly bounded in ${\mathcal S}_a$. Hence, it follows that $u$ is
a bounded function which vanishes at infinity and ${\mathcal S}_a$
is uniformly bounded in $C^{1,\alpha}(\R^3)$ (take $3<m<6$ to get
this embedding). Actually $u$ has further regularity as, using again
the equation, the boundedness of $u$, Calderon-Zygmund estimate, as
well as the Hardy-Littlewood-Sobolev inequality, we have
\begin{align*}
\|u\|_{W^{2,m}}&\leq C\|f\|_{L^m} \leq C\|u\|_{L^m}+C\| (|x|^{-1} * u^2) u\|_{L^m} \\
&\leq C+C\| |x|^{-1} * u^2\|_{L^m}\leq C,
\end{align*}
so that $u\in W^{2,m}(\R^3)$, for every $m\geq 2$. Let us now show that the limit
$u(x)\to 0$ as $|x|\to \infty$ holds uniformly for $u\in{\mathcal
S}_a$. Assuming by contradiction that $u_m(x_m)\geq\sigma>0$ along
some sequences $(u_m)\subset {\mathcal S}_a$ and $(x_m)\subset\R^3$
with $|x_m|\to\infty$, shifting $u_m$ as $v_m(x)=u_m(x+x_m)$ it follows
that $(v_m)$ is bounded in $H^1 (\R^3)\cap L^\infty(\R^3)$ and it
converges, up to a subsequence to a function $v$, weakly in
$H^1(\R^3)$ and locally uniformly in $C(\R^3)$. If $u$ denotes the
weak limit of $u_m$, we also claim that $u,v$ are both
solutions to equation~\eqref{eq:limiting}, which are nontrivial
as follows from (local) uniform convergence and
$u_m(0)\geq u_m(x_m)\geq\sigma$ (since $0$ is a global
maximum for $u_m$) and $v_m(0)=u_m(x_m)\geq\sigma$. To see that $u,v$
are solutions to~\eqref{eq:limiting}, set
\[
\varphi_m(x)=\int_{\R^3}W(x-y)u_m^2(y) dy,
\qquad
\varphi(x)=\int_{\R^3}W(x-y)u^2(y) dy.
\]
Let us show that $\varphi_m(x)\to\varphi(x)$, as $m\to\infty$, for any fixed $x\in\R^3$. Indeed,
we can write $\varphi_m(x)-\varphi(x)=I^1_m(\rho)+I^2_m(\rho)$ for any $m\geq 1$ and any $\rho>0$,
where we set
\begin{align*}
I_m^1(\rho)&=\int_{B_\rho(0)}W(x-y)(u_m^2(y)-u^2(y)) dy,\\
I_m^2(\rho)&=\int_{\R^3\setminus B_\rho(0)}W(x-y)(u_m^2(y)-u^2(y)) dy.
\end{align*}
Fix $x\in\R^3$ and let $\eps>0$. Choose $\rho_0>0$ sufficiently large that
\begin{equation}
    \label{Im2}
I_m^2(\rho_0)\leq\int_{\R^3\setminus B_{\rho_0}(0)}\frac{C}{|y|-|x|}|u_m^2(y)-u^2(y)| dy\leq
\frac{C}{\rho_0-|x|}<\frac{\eps}{2}.
\end{equation}
On the other hand, by the uniform local convergence of $u_m$ to $u$ as $m\to\infty$
and H\"older inequality, for some $1<r<3$
\begin{align}
    \label{Im1}
I^1_m(\rho_0) &\leq C\|u_m-u\|_{L^{r'}(B_{\rho_0}(0))}\Big(\int_{B_{\rho_0}(0)}\frac{1}{|x-y|^{r}}|u_m(y)+u(y)|^r dy\Big)^{1/r} \\
&\leq C\|u_m-u\|_{L^{r'}(B_{\rho_0}(0))}\Big(\int_{B_{\rho_0}(x)}\frac{1}{|y|^{r}}dy\Big)^{1/r} \notag \\
& \leq C\|u_m-u\|_{L^{r'}(B_{\rho_0}(0))}<\frac{\eps}{2}  \notag
\end{align}
for all $m$ sufficiently large, where $r'$ denotes the conjugate exponent of $r$. The bound $r<3$ ensures that
the singular integral which appears in the second inequality is finite.
Combining~\eqref{Im2} with~\eqref{Im1} concludes the proof of the pointwise convergence of $\varphi_m$ to $\varphi$.
For all $\eta \in C_0^\infty(\mathbb{R}^3)$ and any measurable set $E$, we observe that for any $q \in [2,6)$ and $p^{-1}+q^{-1}=1$
\begin{multline*}
\left| \int_E \varphi_m(x) u_m(x) \eta(x) \, dx \right| \leq
\left( \int_E |\eta(x)|^p \, dx \right)^{1/p} \left( \int_E |\varphi_m(x) u_m(x)|^q \, dx \right)^{1/q} \\
\leq  \|\eta\|_{L^p(E)} \left( \int_{\mathbb{R}^3} |u_m(x)|^6 \, dx \right)^{1/6} \left( \int_{\mathbb{R}^3} |\varphi_m(x)|^{\frac{6q}{6-q}}\, dx \right)^{\frac{6-q}{6q}},
\end{multline*}
and we observe that $6q/(6-q) \geq 3$ since $q \geq 2$. Since we already know that $\{\varphi_m\}$ is bounded in any $L^r$ with $r \geq 3$, we conclude that for some constant $C>0$
\[
\left| \int_E \varphi_m(x) u_m(x) \eta(x) \, dx \right| \leq
C  \|\eta\|_{L^p(E)},
\]
and the last term can be made arbitrarily small by taking $E$ of small measure. Since the support of $\eta$ is a compact set and $\varphi_m u_m \eta \to \varphi u \eta$ almost everywhere, the Vitali Convergence Theorem implies
\begin{equation*}
\lim_{m \to +\infty} \int_K \varphi_m u_m\eta\, dx = \int_K \varphi u \eta\, dx,
\end{equation*}
for all $\eta\in C^\infty_c(\R^3)$ with compact support $K$. This
concludes the proof that $u,v$ are nontrivial
solutions to~\eqref{eq:limiting}. It follows that, for any $m$ and $k$,
\begin{align*}
J_a(u_m) &= J_a(u_k)=\frac{1}{4} \int_{\R^3} (|\nabla u_m|^2 + a u^2_m)dx, \\
J_a(u) &\geq J_a(z)=m_a, \\
J_a(v) &\geq  J_a(z)=m_a,
\end{align*}
for all $z\in {\mathcal S}_a$. On the other hand, for any $R>0$ and
$m\geq 1$ with $2R\leq |x_m|$,
\begin{align*}
m_a =J_a(u_m) &\geq \frac{1}{4} \liminf_m\int_{B_R(0)} (|\nabla u_m|^2 + a u^2_m)dx
+\frac{1}{4}\liminf_m \int_{B_R(0)} (|\nabla v_m|^2 + a v^2_m)dx\\
&\geq J_a(u)+J_a(v)-\eps\geq 2m_a-o(1)
\end{align*}
as $R\to\infty$, which yields a contradiction for $R$ large enough. Hence the conclusion follows. Let us now prove that
\begin{equation}
\label{azerocon}
\lim_{|x|\to\infty}\varphi(x)=0,\qquad\varphi(x)=\int_{\R^3}W(x-y)u^2(y)\, dy.
\end{equation}
Notice that, for $\eps>0$ small, $|x|^{-1}$ can be written as the sum of
$|x|^{-1}\chi_{B(0,1)}\in L^{3-\eps}(\R^3)$ and $|x|^{-1}\chi_{\R^3\setminus B(0,1)}\in
L^{3+\eps}(\R^3)$. In particular~\eqref{tozeroconv} of Lemma~\ref{lemma:2.12}
is fulfilled for $u\in H^1(\R^3)$.
Then, since $W\leq C|x|^{-1}$, the quantity $\sup_{|x| \geq R} |W*u^2|$ can be made arbitrarily
small by choosing $R$ large enough.
In light of \eqref{azerocon}, let $R_a>0$ such that
$\varphi(x)\leq \frac{a}{2}$, for any $|x|\geq R_a$. As a consequence,
\begin{equation}
    \label{toexpdec}
-\Delta u(x)+\frac{a}{2}u(x)\leq 0,\qquad\text{for $|x|\geq R_a$}.
\end{equation}
It is thus standard to see that this yields the exponential decay of
$u$, with uniform decay constants in ${\mathcal S}_a$. We can
finally conclude the proof. Let $(u_n)$ be any sequence in
${\mathcal S}_a$. Up to a subsequence it follows that $(u_n)$
converges weakly to a function $u$ which is also a solution to
equation~\eqref{eq:limiting}. If $\D$ is the function defined
in~\eqref{defDD}, we immediately get
\begin{equation}
\label{quasistron}
\lim_{n \to +\infty}\int_{\R^3}|\nabla u_n|^2+au_n^2-\D(u_n)=0
=\int_{\R^3}|\nabla u|^2+au^2-\D(u).
\end{equation}
Hence the desired strong convergence of $(u_n)$ to $u$ in $H^1(\R^3)$ follows once we prove that
$\D(u_n)\to \D(u)$, as $n\to\infty$. In view of the uniform
exponential decay of $u_n$ it follows that $u_n\to u$ strongly in $L^{12/5}(\R^3)$, as $n\to\infty$.
Taking into account that $(u_n)$ is bounded in $H^1(\R^3)$ and that $W$ is even, we get
the inequality
\begin{equation*}
|\D(u_n)-\D(u)|\leq \sqrt{\tilde\D(||u_n|^2-|u|^2|^{1/2})}\sqrt{\tilde\D((|u_n|^2+|u|^2)^{1/2})},\quad n\in{\mathbb N},
\end{equation*}
being $\tilde \D$ defined as the operator $\D$ with the Coulomb kernel $|x|^{-1}$ in place of $\D$.
In fact, taking into account \cite[Theorem 9.8, p.250]{ll} and since $W$ is even, we have
\begin{align*}
|\D(u_n)-\D(u)|&=\Big|\int_{\R^6}W(x-y)|u_n(x)|^2|u_n(y)|^2dxdy-\int_{\R^6}W(x-y)|u(x)|^2|u(y)|^2dxdy\Big|  \\
&=\Big|\int_{\R^6}W(x-y)|u_n(x)|^2|u_n(y)|^2dxdy+\int_{\R^6}W(x-y)|u_n(x)|^2|u(y)|^2dxdy  \\
& -\int_{\R^6}W(x-y)|u(x)|^2|u_n(y)|^2dxdy-\int_{\R^6}W(x-y)|u(x)|^2|u(y)|^2dxdy\Big| \\
&=\Big|\int_{\R^6}W(x-y)(|u_n(x)|^2-|u(x)|^2)(|u_n(y)|^2+|u(y)|^2)dxdy\Big| \\
&\leq \int_{\R^6}W(x-y)||u_n(x)|^2-|u(x)|^2|||u_n(y)|^2+|u(y)|^2|dxdy \\
&\leq C\int_{\R^6}|x-y|^{-1}||u_n(x)|^2-|u(x)|^2|||u_n(y)|^2+|u(y)|^2|dxdy \\
&\leq C \sqrt{\tilde\D(||u_n|^2-|u|^2|^{1/2})}\sqrt{\tilde\D((|u_n|^2+|u|^2)^{1/2})}.
\end{align*}
Then, by Hardy-Littlewood-Sobolev inequality and H\"older's inequality, it follows that
\begin{equation*}
|\tilde \D(u_n)-\tilde\D(u)|^2 \leq  C \|\,||u_n|^2-|u|^2|^{1/2}\|_{L^\frac{12}{5}}^{4}\|
\,(|u_n|^2+|u|^2)^{1/2}\|_{L^\frac{12}{5}}^{4}
\leq  C \|u_n-u\|_{L^\frac{12}{5}}^{2}.
\end{equation*}
As a consequence
\[
\D(u_n)=\D(u)+o(1),\quad\text{as $n\to\infty$},
\]
which concludes the proof in light of formula~\eqref{quasistron}.
\end{proof}

\medskip

\section{The penalization argument}

Throughout this and the following sections we shall mainly used the arguments of~\cite{cjs}
highlighting the technical steps where the Hartree nonlinearity is involved in place of
the local one. For the sake of self-containedness and for the reader's convenience we develop
the arguments with all the detail.

\vskip2pt
For any set $\Omega \subset \R^3$ and $\ge >0$, let $\Omega_\ge = \{
x \in \R^3: \ge x \in \Omega \}$.

\subsection{Notations and framework}

The following lemmas, taken
from \cite{cjs} show that the norm in $H_\varepsilon$ is locally
equivalent to the standard $H^1$ norm.
\begin{lemma} \label{equi-norm}
  Let $K \subset \R^3$ be an arbitrary,
  fixed, bounded domain. Assume that $A$ is bounded on $K$ and $0 <
  \alpha \leq V \leq \beta$ on $K$ for some $\alpha$, $\beta >0$. Then,
  for any fixed $\e \in [0,1],$ the norm
\[
\|u\|_{K_\e}^2 = \int_{K_\e}\bigg|\left(\frac{1}{\iu}\nabla -
  A_\e(y)\right) u \bigg|^2 + {V}_\e (y) |u|^2 dy
\]
is equivalent to the usual norm on $H^1(K_\e, \C)$. Moreover these
equivalences are uniform, namely there exist constants $c_1, c_2 >0$ independent
of $\e \in [0,1]$ such that
\[
c_1 \|u\|_{K_\e} \leq \|u\|_{H^1(K_\e, \C)} \leq c_2 \|u\|_{K_\e}.
\]
\end{lemma}

\begin{corollary} \label{equiv-norm-consequence}
   Retain the setting of~Lemma \ref{equi-norm}. Then the following facts hold.
\begin{enumerate}
\item[(i)] If $K$ is compact, for any $\e \in (0,1]$ the norm
  \[
  \|u\|_K^2 := \int_K \left| \left( \frac{1}{\iu}\nabla - A_\e (y)
    \right )u\right|^2 + {V}_\ge (y) |u|^2 dy
  \]
  is uniformly equivalent to the usual norm on $H^1(K, \C)$.
\item[(ii)] For $A_0 \in \R^3$ and $b>0$ fixed, the norm
  \[
  \|u\|^2 := \int_{\R^3}\left| \left( \frac{1}{\iu}\nabla - A_0 \right)
    u\right|^2 + b|u|^2 dy
  \]
  is equivalent to the usual norm on $H^1(\R^3, \C)$.
\item[(iii)] If $(u_{\ge_n}) \subset H^1(\R^3, \C)$ satisfies
  $u_{\ge_n} = 0$ on $\R^3 \setminus K_{\ge_n}$ for any $n \in
  \mathbb{N}$ and $u_{\ge_n} \to u$ in $H^1(\R^3, \C)$ then
  $\|u_{\ge_n} - u\|_{{\ge_n}} \to 0$ as $n \to \infty.$
\end{enumerate}
\end{corollary}

For future reference we recall the following \textit{Diamagnetic
  inequality}: for every $u \in H_\e$,
\begin{equation}\label{eq:2.6}
  \left| \left(\frac{\nabla}{\iu} - A_\e \right) u\right| \geq \big|
  \nabla |u|\big|, \quad \mbox{a.e. in $\RN$. }
\end{equation}
See \cite{EL} for a proof. As a consequence of \eqref{eq:2.6}, $|u|
\in H^1(\R^3, \R)$ for any $u \in H_{\e}.$
\vskip4pt

For any $u \in H_\e,$ let us set
\begin{equation}
{\mathcal F}_\ge (u) = \frac{1}{2}\int_{\RN}| D^\ge u|^2 +V_\e(x) |u|^2
\, dx - \frac{1}{4}\int_{\RN \times \RN} W(x-y)|u(x)|^2 |u(y)|^2 dxdy
\end{equation}
where we set $D^\ge = (\frac{\nabla}{\iu} - A_\e)$. Define, for all $\eps>0$,
\begin{equation*}
\chi_\e(y) =
\begin{cases}
0 &\text{if $y \in O_\e$}, \\
\e^{-6 / \mu} &\text{if $y \notin O_\e$},
\end{cases}
\quad \chi^i_\e(y) =
\begin{cases}
0 & \text{if $y \in (O^i)_\e$}, \\
\e^{-6 / \mu} &\text{if $y \notin (O^i)_\e$},
\end{cases}
\end{equation*}
and
\begin{equation}
Q_\e(u) = \Big (\intr \chi_\e |u|^2 dx -  1\Big )^{\frac{5}{2}}_+, \quad
 Q^i_\e(u) = \Big (\intr \chi^i_\e |u|^2 dx - 1\Big
)^{\frac{5}{2}}_+.
\end{equation}
The functional $Q_\e$ will act as a penalization to force the concentration phenomena
of the solution to occur inside $O$. In particular, we remark that the penalization terms vanish on elements
whose corresponding $L^\infty$-norm is sufficiently small. This device
was firstly introduced in~\cite{BW2}. Finally we define the
functionals $\Gamma_\e,\Gamma_\e^1,\ldots,\Gamma_\e^k : H_\e \to \R$ by setting
\begin{equation}
\Gamma_\e(u) = {\mathcal F}_\ge(u) + Q_\e(u), \,\,\,  \Gamma_\e^i(u) =
{\mathcal F}_\ge(u) + Q^i_\e(u),\quad  i = 1,\ldots ,k.
\end{equation}
It is easy to check, under our assumptions, and using the
Diamagnetic inequality~\eqref{eq:2.6}, that the functionals
$\Gamma_\e$ and $\Gamma_\e^i$ are of class $C^1$ over $H_\e.$ Hence, a critical point of
${\mathcal F}_\ge $ corresponds to a solution of
\eqref{eq:Hartree-scaled}. To find solutions of
\eqref{eq:Hartree-scaled} which {\it concentrate} in $O$ as $\e \to
0,$ we shall look for a critical point of $\Gamma_\e$ for which
$Q_\e$ is zero.
\medskip

Let
\[
\M = \bigcup_{i=1}^k\M^i, \quad O = \bigcup_{i=1}^k O^i
\]
and for any set $B \subset \RN$ and $\alpha > 0,$ $B^\delta = \{x
\in \RN : \operatorname{dist}(x,B) \le \delta\}$ and set
\[ \delta = \frac{1}{10} \min \left\{\operatorname{dist}(\M,\RN
\setminus O),\, \min_{ i\ne j} \operatorname{dist}(O_i,O_j),\,
\operatorname{dist}(O,\Z) \right\}.
\]
We fix a $\beta \in (0,\delta)$ and a cutoff $\p \in
C_0^\infty(\RN)$ such that $0 \le \p \le 1,$ $\p(y) = 1$ for $|y|
\le \beta$ and $\p(y) = 0$ for $|y| \geq 2\beta$. Also, setting
$\p_\e(y) = \p(\e y)$  for each $x_i \in (\M^i)^\beta$ and $U_i \in
{\mathcal S}_{m_i},$ we define
\[
U_\e^{x_1,\dots,x_k}(y) = \sum_{i=1}^k e^{iA(x_i)(y-
\frac{x_i}{\e})} \p_\e \left(y-\frac{x_i}{\e} \right) U_i
\left(y-\frac{x_i}{\e} \right).
\]
We will find a solution, for sufficiently small $\e
> 0,$  near the set
\[ X_\e = \{U_\e^{x_1\dots,x_k}(y)  : \text{$x_i \in (\M^i)^\beta$ and $U_i \in {\mathcal S}_{m_i}$ for each $i = 1,\dots,k$}\}.
\]
For each $i \in \{1,\dots,k\}$ we fix an arbitrary $x_i \in \M^i$
and an arbitrary $U_i \in {\mathcal S}_{m_i}$ and we define
\[
\mathcal{W}^i_\e(y) = e^{iA(x_i)(y -\frac{x_i}{\e})}\p_\e
\left(y-\frac{x_i}{\e}\right) U_i \left(y-\frac{x_i}{\e} \right).
\]
Setting
\[
\mathcal{W}^i_{\e,t}(y) = e^{iA(x_i)(y -\frac{x_i}{\e})} \p_\e
\left(y-\frac{x_i}{\e} \right) U_i \left(\frac{y}{t}-\frac{x_i}{\e
t} \right),
\]
 we see that
$$
\lim_{t \to 0}\Vert \mathcal{W}^i_{\e,t}
\Vert_\e = 0,\qquad
\Gamma_\e(\mathcal{W}^i_{\e,t}) = {\mathcal F}_\ge(\mathcal{W}^i_{\e,t}),\quad t\geq 0.
$$
In the next Proposition we shall show that there exists
$T_i>0$ such that $\Gamma_\e(\W_{\e,T_i}^i) < - 2$
for any $\e >0$ sufficiently small. Assuming this holds true, let  $\gamma_\e^i(s) = \W^i_{\e,s}$ for $s
> 0$ and $\gamma_\e^i(0) = 0.$  For $s = (s_1,\dots,s_k) \in T =
[0,T_1] \times \ldots \times [0,T_k]$ we define
\[
\gamma_\e(s) = \sum_{i=1}^k \W^i_{\e,s_i} \quad \mbox{ and } \quad
 D_\e = \max_{s \in T } \Gamma_\e (\gamma_\e(s)).
 \]
Finally for each $i \in \{1,\dots,k\},$ let $E_{m_i} = L^c_{m_i}(U)$
for $U \in S_{m_i}$. Here $L^c_a$ is, for any $a>0$, the Euler functional associated to (\ref{eq:limiting}) in which solutions are considered as complex-valued.

\subsection{Energy estimates and Palais-Smale condition}
In what follows, we set
\[
E_m =
\min_{i \in \{1, \ldots, k\}}E_{m_i},\quad\,\,
E = \sum_{i=1}^k E_{m_i}.
\]
For a set $A \subset H_\e$  and $\alpha> 0$,
we let $A^\alpha = \{u \in  H_\e :   \| u - A \|_\ge \le
\alpha\}.$

\begin{proposition}\label{prop2}
There results
\begin{itemize}
\item[(i)] $\displaystyle \lim_{\e \to 0}D_\e = E,$
\item[(ii)] $\displaystyle  \limsup_{\e \to 0} \max_{s \in
\partial T} \Gamma_\e (\gamma_\e(s)) \le \tilde{E} = \max\{E
- E_{m_i} \ | \ i =1,\dots,k\} < E,$
\item[(iii)] for each $d> 0,$
there exists $\alpha >0$ such that for sufficiently small $\e > 0,$
\[
\Gamma_\e (\gamma_\e (s)) \geq D_{\e}- \alpha \text{ implies that }
\gamma_\e (s) \in X_\e^{d/2}.
\]
\end{itemize}
\end{proposition}
\begin{proof} Since $\operatorname{supp}(\gamma_\e(s)) \subset \M_\e^{2\beta}$ for each $s
\in T,$ it follows that
$$
\Gamma_\e(\gamma_\e(s))= {\mathcal F}_\ge(\gamma_\e(s)) =
\sum_{i=1}^k {\mathcal F}_\ge
(\gamma_\e^i(s)).
$$
Arguing as in \cite[Proposition 3.1]{cjs}, we claim that for each
$i \in \{1, \dots, k\}$
\begin{equation}\label{eq:3.1} \lim_{\e \to 0} \int_{\R^3}\left|
\bigg( \frac{\nabla}{\iu} - A_\e(y)\bigg) W_{\e,s_i}^i\right|^2 dy =
s_i\int_{\R^3}|\nabla U_i|^2 dy.
\end{equation}
Using the exponentially decay of $U_i$ we have that, as $\e \to 0$,
\begin{equation}\label{eq:3.2}
\int_{\R^3}V_\e(y) |W^i_{\e,s_i}|^2 dy = \int_{\R^3}m_i \, \Big|
U_i \big( \frac{y}{s_i} \big) \Big|^2 dy+o(1) = m_i s_i^3
\int_{\R^3}|U_i|^2 dy+o(1),
\end{equation}
and, as $\e \to 0$,
\begin{multline*}
\int_{\R^{6}} W(x-y)|W^i_{\e,s_i}(x)|^2 |W^i_{\e,s_i}(y)|^2 dx dy \\
=\int_{\R^{6}} W(x-y)
 \Big| U_i \big(\frac{y}{s_i}\big)\Big|^2 \Big| U_i \big(\frac{x}{s_i}\big) \Big|^2  dx dy+o(1)
\\
= s_i^5 \int_{\R^{6}} W(x-y)|U_i(x)|^2|U_i(y)|^2 dxdy+o(1).
\end{multline*}
Thus, from the above limit and from~\eqref{eq:3.1}, \eqref{eq:3.2},
we derive
\begin{align*}
{\mathcal F}_\ge(\gamma_\e^i(s_i)) &=
 \frac{1}{2}\int_{\R^3} \bigg| \bigg(
\frac{\nabla}{\iu} - A_\e (y) \bigg) \gamma_\e^i(s_i) \bigg|^2 dy  +
V_\e(y)|\gamma_\e^i(s_i)|^2 dy
\\
&- \frac{1}{4} \int_{\R^{6}} W(x-y)|\gamma_\e^i(s_i)|^2 |\gamma_\e^i(s_i)|^2 dx dy \\
&= \frac{s_i}{2}\intr |\nabla U_i|^2 dy +  \frac{1}{2} s_i^{3} m_i\intr
|U_i|^2dy \\
&- \frac{1}{4} s_i^{5}  \int_{\R^{6}}
W(x-y)|U_i(x)|^2|U^i(y)|^2 dx\, dy + o(1).
\end{align*}
Using the Pohozaev identity~\eqref{eq:poho} as well as the relation
$$
\int_{\R^{3}} |\nabla U_i|^2dy + m_i \int_{\R^{3}} |U_i|^2dy = \int_{\R^{6}} W(x-y) |U_i(x)|^2 |U_i(y)|^2 \, dx dy,
$$
we see that
\[
{\mathcal F}_\ge(\gamma_\e^i(s_i)) = \left( \frac{ s_i}{6} + \frac{s_i^{3}}{2} - \frac{ 2s_i^{5}}{6} \right) m_
i\intr |U_i|^2 dy + o(1).
\]
Also
\[
\max_{t \in [0,\infty)} \left( \frac{t}{6} +
\frac{t^{3}}{2} - \frac{ 2t^{5}}{6}\right) m_i\intr |U_i|^2 dy
=E_{m_i}.
\]
At this point we deduce that (i) and (ii) hold. Clearly also the
existence of a $T_i >0$ such that $\Gamma_\e(W_{\e,T_i}^i) <-2$ is
justified. To conclude we just observe that, setting
\begin{equation*}
g(t) =
\frac{t}{6} +
\frac{t^{3}}{2} - \frac{ 2t^{5}}{6}
\end{equation*}
the derivative $g'(t)$ of $g(t)$
is positive for $t \in (0,1)$, negative for $t \in (1,+\infty)$, and vanishes at $t=1$.
We conclude by observing that $g''(1)  < 0.$
\end{proof}

\vskip2pt

Let us define
\[
\Phi^i_\ge = \left\{ \gamma \in C([0,T_i], H_\ge) :\, \gamma
(0)=\gamma_\e^i(0), \ \gamma (T_i) = \gamma_\e^i(T_i) \right\}
\]
and
\[
C^i_\ge = \inf_{\gamma \in \Phi^i_\ge} \max_{s_i \in [0,T_i]}
\Gamma^i_\ge (\gamma(s_i)).
\]
\vspace{1pt}

\begin{proposition}\label{upper}
For the level $C^i_\ge$ defined before, there results
\[
\liminf_{\ge \to 0} C^i_\ge \geq E_{m_i}.
\]
In particular, $\lim_{\ge \to 0} C^i_\ge = E_{m_i}$.
\end{proposition}
\begin{proof}
The proof of this lemma is  analogous to that of Proposition 3.2 in
\cite{cjs}.
\end{proof}

Next we define, for every $\alpha \in \mathbb{R}$, the sub-level
\[
\Gamma_\ge^\alpha = \{ u \in H_\ge :\,  \Gamma_\ge(u)
\leq \alpha\}.
\]
\begin{proposition}\label{prop40}
Let $(\ge_j)$ be such that $\lim_{j \to \infty}\ge_j = 0$ and
$(u_{\ge_j}) \in X_{\ge_j}^{d}$ such that
\begin{equation}\label{ass}
\lim_{j \to \infty}\Gamma_{\ge_j}(u_{\ge_j}) \le E \mbox{ and }
\lim_{j \to \infty}\Gamma_{\ge_j}^\prime(u_{\ge_j}) = 0.
\end{equation}
Then, for sufficiently small $d
> 0,$ there exist, up to a subsequence, $(y^i_j ) \subset \R^3$, $i
=1,\dots,k$, points $x^i\in \M^i$ (not to be confused with
the points $x_i$ already introduced), $U_i \in S_{m_i}$ such that
\begin{equation}\label{ass1}
\lim_{j \to \infty} |\ge_j y^i_j - x^i| = 0  \text{ and } \lim_{j \to
\infty} \left\Vert u_{\varepsilon_j} - \sum_{i=1}^k e^{iA_\ge(y^i_j)(\cdot
-y_j^i)}\p_{\ge_j}(\cdot - y^i_j) U_i(\cdot - y^i_j) \right\Vert_{\ge_j} = 0.
\end{equation}
\end{proposition}
\begin{proof} For simplicity we write $\ge$
instead of $\ge_j.$ From Proposition \ref{prop:compact}, we know that the
$S_{m_i}$ are compact. Then there exist $Z_i \in S_{m_i}$ and
$(x_\varepsilon^i) \subset (\M^i)^\beta$, $x^i \in (\M^i)^\beta$
for $i =1,\ldots,k$, with $x^i_\varepsilon \to x^i$ as
$\varepsilon \to 0$ such that, passing to a subsequence still
denoted $(u_\ge)$,
\begin{equation}
\label{401} \left\| u_\ge - \sum_{i=1}^k e^{iA(x_\varepsilon^i)(\cdot-
\frac{x^i_\ge}{\ge})}\p_{\ge}(\cdot - x^i_\ge/\ge)Z_i(\cdot - x^i_\ge/\ge)
\right\|_\ge \le 2d
\end{equation}
for small $\ge > 0.$ We set $u_{1,\ge} = \sum_{i=1}^k \p_\ge(\cdot -
x_\ge^i/\ge)u_\ge$ and $u_{2,\ge} = u_\ge -u_{1,\ge}$. As a first step in
the proof of the Proposition we shall prove that
\begin{equation}
\label{402}\Gamma_\ge(u_\ge) \geq \Gamma_\ge(u_{1,\ge}) +
\Gamma_\ge(u_{2,\ge}) + O(\ge).
\end{equation}
Suppose there exist $y_\ge \in \bigcup_{i=1}^{k}
B(x^i_\ge/\ge,2\beta/\ge) \setminus B(x^i_\ge/\ge,\beta/\ge)$ and $R>0$
satisfying
\[
\liminf_{\ge \to 0}\int_{B(y_\ge,R)}|u_\ge|^2 dy
> 0
\]
which means that
\begin{equation}\label{4022}
\liminf_{\ge \to 0}\int_{B(0,R)}|v_\ge|^2 dy
> 0
\end{equation}
where $v_\ge(y) = u_\ge(y+y_\ge)$. Taking a subsequence, we can assume
that $\ge y_\ge \to x_0 $ with $x_0$ in the closure of
$\bigcup_{i=1}^k B(x^i,2\beta) \backslash B(x^i,\beta)$. Since
(\ref{401}) holds, $(v_\ge)$ is bounded in $H_\ge$. Thus, since
$\tilde{m} >0,$ $(v_\ge)$ is bounded in $L^2(\R^3,\C)$ and using the
Diamagnetic inequality and the Hardy-Sobolev inequality (see also the proof of Proposition \ref{prop:compact}) we deduce that $(v_\ge)$ is
bounded in $L^{m}(\R^3, \C)$ for any $m<6$. In particular, up to a subsequence,
$v_\ge \to \mathcal{W} \in L^{m}(\R^3, \C)$ weakly. Also by Corollary
\ref{equiv-norm-consequence} i), for any compact $K \subset \R^3$,
$(v_\ge)$ is bounded in $H^1(K, \C)$. Thus we can assume that $v_\ge
\to \mathcal{W}$ in $H^1(K, \C)$ weakly for any $K \subset \R^3$ compact,
strongly in $L^{m}(K,\mathbb{C})$. Because of (\ref{4022}) $\mathcal{W}$ is
not the zero function. Now, since $\lim_{\ge \to 0}\Gamma'_\ge(u_{\ge})
=0,$ $\mathcal{W}$ is a non-trivial solution of
\begin{equation}\label{3.100}
- \Delta  \mathcal{W}  - \frac{2}{\iu} A(x_0) \cdot \nabla \mathcal{W} + |A(x_0)|^2 \mathcal{W} +
V(x_0) \mathcal{W} = \left(W*|\mathcal{W}|^2 \right) \mathcal{W}.
\end{equation}
From \eqref{3.100} and since $\mathcal{W} \in L^{m}(\R^3, \C)$ we readily
deduce, using Corollary \ref{equiv-norm-consequence} ii) that $\mathcal{W} \in
H^1(\R^3, \C).$ \medskip

Let $ \omega(y)= e^{- i A(x_0) y} \mathcal{W}(y)$. Then $\omega$ is a non
trivial solution of the  complex-valued equation
\[
-\Delta  \omega + V(x_0) \omega = (W*|\omega|^2) \omega.
\]
For $R>0$ large we have
\begin{equation}\label{happy}
\int_{B(0, R)} \left| \left( \frac{\nabla}{\iu} - A(x_0) \right) \mathcal{W}
\right|^2 dy \geq \frac{1}{2}\int_{\R^3} \left|
\left(\frac{\nabla}{\iu} - A(x_0)\right) \mathcal{W} \right|^2 dy
\end{equation}
and thus, by the weak convergence,
\begin{align}\label{nnn}
\liminf_{\ge \to 0} \int_{B(y_\ge, R)} | D^\ge u_\ge|^2 dy &=
\liminf_{\ge \to 0} \int_{B(0, R)} \left| \left( \frac{\nabla}{\iu} -
A_\ge(y+ y_\ge) \right)
v_\ge \right|^2 dy \notag \\
&\geq \int_{B(0, R)} \left| \left( \frac{\nabla}{\iu} - A(x_0)
\right) \mathcal{W}
\right|^2 dy \notag \\
&\geq \frac{1}{2} \int_{\R^3} \left| \left(\frac{\nabla}{\iu} -
A(x_0)\right) \mathcal{W} \right|^2 dy = \frac{1}{2} \int_{\R^3} |\nabla
\omega |^2 dy.
\end{align}
It follows from Lemma \ref{equiv-min} that $E_a > E_b$ if $a >b$ and using
Lemma \ref{levels} we have $L^c_{V(x_0)} (\omega) \geq E^c_{V(x_0)}
= E_{V(x_0)} \geq E_m$ since $V(x_0) \geq m$. Thus from (\ref{nnn})
and Lemma \ref{equiv-min} we get that
\begin{equation} \label{eq:no}
\liminf_{\ge \to 0} \int_{B(y_\ge, R)} | D^\ge u_\ge|^2 dy \geq
\frac{3}{2} L^c_{V(x_0)} (\omega) \geq \frac{3}{2}E_m>0 .
\end{equation}
which contradicts (\ref{401}), provided $d>0$ is small enough.
Indeed, $x_0 \neq x^i$, $\forall i \in \{1,\ldots,k\}$ and the $Z_i$
are exponentially decreasing.
\medskip

Since such a sequence $(y_{\varepsilon})$ does not exist, we deduce
from \cite[Lemma I.1]{L} that
\begin{equation}
\label{404} \limsup_{\ge \to 0}\int_{ \bigcup_{i=1}^k B(x_\ge^i/\ge,
2\beta/\ge) \setminus B(x^i_\ge/\ge, \beta/\ge)}|u_\ge|^{5} dy = 0.
\end{equation}
As a consequence, we can derive using the boundedness of $(\|u_\ge\|_2)$ that
\begin{multline*}
\lim_{\ge \to 0} \left\{ \int_{\R^3 \times \R^3} W(x-y)|u_\ge(x)|^2 |u_\ge (y)|^2\, dx\, dy
- \int_{\R^3 \times \R^3} W(x-y)|u_{1,\ge}(x)|^2 |u_{1,\ge} (y)|^2\, dx\, dy \right. \\
\left. -\int_{\R^3 \times \R^3} W(x-y)|u_{2,\ge}(x)|^2 |u_{2,\ge} (y)|^2\, dx\, dy \right\}
=0.
\end{multline*}
At this point writing
\begin{align*}
\Gamma_\ge(u_\ge) &= \Gamma_\ge(u_{1,\ge}) + \Gamma_\ge(u_{2,\ge}) \\
&+\sum_{i=1}^k \int\limits_{B(x^i_\ge /\ge, 2\beta/\ge) \setminus B(x^i_\ge /\ge,
\beta/\ge)}
 \p_\ge (y-x^i_\ge /\ge)(1- \p_\ge (y-x^i/\ge))|D^{\ge} u_\ge|^2 \\
 & \quad {}+ V_\ge \p_\ge (y-x^i_\ge/\ge) (1- \p_\ge (y-x^i_\ge /\ge))|u_\ge|^2 dy \\
& {}- \frac{1}{4}  \int_{\R^3 \times \R^3} W(x-y)|u_\ge(x)|^2 |u_\ge (y)|^2\, dx\, dy  \\
 & {}-   \frac{1}{4} \int_{\R^3 \times \R^3} W(x-y)|u_{1,\ge}(x)|^2
|u_{1,\ge} (y)|^2\, dx\, dy \\
&{}-  \frac{1}{4} \int_{\R^3 \times \R^3}
W(x-y)|u_{2,\ge}(x)|^2 |u_{2,\ge} (y)|^2\, dx\, dy
+ o(1),
\end{align*}
as $\ge \to 0$ this shows that the inequality (\ref{402}) holds. We
now estimate $\Gamma_\ge(u_{2,\ge})$. We have
\begin{align} \label{405}
  \Gamma_\ge(u_{2, \ge})
&\geq{\mathcal F}_\ge(u_{2, \ge}) \\
&=  \frac{1}{2}
\int_{\R^3}|D^{\ge} u_{2, \ge}|^2 + \tilde{V}_\ge |u_{2, \ge}|^2 dy -
\frac{1}{2}\int_{\R^3}(\tilde{V}_\ge - V_\ge) |u_{2, \ge}|^2dy \nonumber\\
 &\quad{}- \frac{1}{4} \int_{\R^3 \times \R^3 }W(x-y) |u_{2, \ge}(x)|^2 |u_{2, \ge} (y)|^2 \, dx\,
dy  \nonumber \\
&\geq \frac{1}{2}\Vert u_{2, \ge}\Vert_\ge^2 -
\frac{\tilde{m}}{2}\int_{\R^3 \setminus O_\ge^i}|u_{2, \ge}|^2 dy \nonumber \\
& \quad {}-
\frac{1}{4} \int_{\R^3 \times \R^3 }W(x-y) |u_{2, \ge}(x)|^2 |u_{2, \ge} (y)|^2 \, dx\,
dy.   \nonumber
\end{align}
Here we have used the fact that $\tilde V_\ge-V_\ge=0$ on $O_\ge^i$ and
$|\tilde V_\ge-V_\ge|\leq \widetilde m$ on $\R^3\setminus O_\ge^i$.
For some $C>0$,
\begin{equation*}
\int_{\R^3 \times \R^3 }W(x-y) |u_{2, \ge}(x)|^2 |u_{2, \ge} (y)|^2 \, dx\,
dy \leq C \|u_{2,\ge}\|_{L^2}^{3} \, \|u_{2,\ge}\|_{H^1}.
\end{equation*}
Since $(u_\ge)$ is bounded, we see from (\ref{401}) that $\Vert u_{2,
\ge} \Vert_\ge \le 4d$ for small $\ge
>0$. Thus taking $d>0$ small enough we have
\begin{equation}
\label{406} \frac{1}{2}\|u_{2, \ge}\|_\ge^2 - \frac{1}{4} \int_{\R^3 \times \R^3 }W(x-y) |u_{2, \ge}(x)|^2 |u_{2, \ge} (y)|^2 \, dx\,
dy  \geq \frac{1}{8} \|u_{2, \ge}\|_\ge^2.
\end{equation}
Now note that ${\mathcal F}_\ge$ is uniformly bounded in $X_\ge^d$
for small $\ge
> 0$, and such is $Q_\ge.$ This implies that for some $C > 0,$
\begin{equation} \label{407}
\int_{\RN \setminus O_\ge} |u_{2, \ge}|^2 dy \le C\ge^{6 / \mu}
\end{equation}
and from (\ref{405})-(\ref{407}) we deduce that $\Gamma_\ge(u_{2,
\ge}) \geq o(1).$
\medskip

Now for $i = 1,\ldots,k,$ we define $u_{1, \ge}^i(y) = u_{1,\ge}(y)$
for $y \in O_\ge^i, $ $u_{1, \ge}^i(y) = 0$ for $y \notin O_\ge^i$.
Also we set $\mathcal{W}_\ge^i (y) = u_{1, \ge}^i(y + x^i_\ge
/\ge).$ We fix an arbitrary $i \in \{1, \ldots, k\}$. Arguing as
before, we can assume, up to a subsequence, that
$\mathcal{W}_{\ge}^i$ converges weakly in $L^{m}(\R^3, \C)$, $m <
6$, to a solution $\mathcal{W}^i \in H^1(\R^3, \C)$ of
\[
- \Delta  \mathcal{W}^i - \frac{2}{\iu} A(x^i) \cdot \nabla \mathcal{W}^i + |A(x^i)|^2 \mathcal{W}^i
+ V(x^i) \mathcal{W}^i = \left( W * \mathcal{W}^i \right) \mathcal{W}^i, \quad y\in \R^3.
\]
We shall prove that $\mathcal{W}_\ge^i$ tends to $\mathcal{W}^i$ strongly in $H_\ge$.
Suppose there exist $R>0$ and a sequence $(z_\ge)$ with $z_\ge \in
B(x^i_\ge /\ge, 2 \beta /\ge)$ satisfying
\[
\liminf_{\ge \to 0} |z_\ge - {\ge}^{-1}x^i_\ge| = \infty \quad \hbox{and}
\quad \liminf_{\ge \to 0} \int_{B(z_\ge , R)} |u^{1,i}_\ge|^2 \, dy
>0.
\]
We may assume that $\ge z_\ge \to z^i \in O^i$ as $\ge \to 0$. Then
$\tilde{\mathcal{W}}_\ge^i(y)= \mathcal{W}_\ge^i (y + z_\ge)$ weakly converges in
$L^{m} (\R^3, \C)$ (for any $m<6$) to $\tilde{\mathcal{W}}^i \in H^1(\R^3, \C)$ which satisfies
\[
- \Delta  \tilde{\mathcal{W}}^i - \frac{2}{\iu} A(z^i) \cdot \nabla \tilde{\mathcal{W}} +
|A(z^i)|^2 \tilde{\mathcal{W}}^i + V(z^i) \tilde{\mathcal{W}}^i = \left( W * \tilde{\mathcal{W}}^i \right) \tilde{\mathcal{W}}^i, \quad y \in \R^3
\]
and as before we get a contradiction. Then using \cite[Lemma I.1]{L}  it follows that
\begin{multline}\label{3.101}
\lim_{\varepsilon \to 0} \int_{\R^3 \times \R^3} W(x-y) |\mathcal{W}_\ge^i(x)|^2 |\mathcal{W}_\ge^i(y)|^2 \, dx\, dy \\
= \int_{\R^3 \times \R^3} W(x-y) |\mathcal{W}^i(x)|^2 |\mathcal{W}^i(y)|^2 \, dx\, dy .
\end{multline}
Then from the weak convergence of $\mathcal{W}_\ge^i$ to $\mathcal{W}^i \neq 0$ in
$H^1(K, \C)$ for any $K \subset \R^3$ compact we get, for any $i \in
\{1, \ldots, k \}$,
\begin{align}
\limsup_{\ge \to 0}\Gamma_\ge(u_{1, \ge}^i)
 &\geq \liminf_{\ge \to 0} {\mathcal F}_\ge(u_{1, \ge}^i)  \nonumber \\
&\geq \liminf_{\ge \to 0}  \frac{1}{2} \int_{B(0, R)}
\bigg|\left(\frac{\nabla}{\iu} - A(\ge y + x^i_\ge)\right)
\mathcal{W}_{\ge}^i\bigg|^2
\notag \\
&{}+ V(\ge y + x^i_\ge )|\mathcal{W}^i_{\ge}|^2 dy -\hfill
\int_{\R^3 \times \R^3} W(x-y) |\mathcal{W}_\ge^i(x)|^2 |\mathcal{W}_\ge^i(y)|^2 \, dx\, dy  \nonumber \\
& \geq  \frac{1}{2} \int_{B(0, R)} \bigg|\left(\frac{\nabla}{\iu} -
A(x^i)\right) \mathcal{W}^i\bigg|^2 +V(x^i)|\mathcal{W}^i|^2 dy \notag \\
& \hfill {}-  \frac{1}{4} \int_{\R^3 \times \R^3} W(x-y) |\mathcal{W}^i(x)|^2 |\mathcal{W}^i(y)|^2 \, dx\, dy. \label{411}
\end{align}
 Since these inequalities hold for any $R >0$ we
deduce, using Lemma \ref{levels}, that
\begin{align}\label{200}
\limsup_{\ge \to 0} \Gamma_\ge (u_{1, \ge}^i) & \geq  \frac{1}{2}
\int_{\R^3} \left| \left(\frac{\nabla}{\iu} - A(x^i)\right)
\mathcal{W}^i \right|^2 dy + \frac{1}{2}\int_{\R^3} V(x^i) |\mathcal{W}^i|^2 dy \notag \\
& {} - \int_{\R^3 \times \R^3} W(x-y) |\mathcal{W}^i(x)|^2 |\mathcal{W}^i(y)|^2 \, dx\, dy \nonumber \\
& =  \frac{1}{2} \int_{\R^3} |\nabla  \omega^i |^2  + V(x^i)
|\omega^i|^2 dy  \nonumber \\
& {} - \frac{1}{4} \int_{\R^3 \times \R^3} W(x-y) |\omega^i(x)|^2 |\omega^i(y)|^2 \, dxdy
 =  L^c_{V(x^i)}(\omega^i) \geq E_{m_i}^c=E_{m_i}
\end{align}
where we have set $\omega^i(y) = e^{-i A(x^i)y} \mathcal{W}^i(y)$. Now by
(\ref{402}),
\begin{align} \label{412}
\limsup_{\ge \to 0} \Big ( \Gamma_\ge (u_{2,\ge}) + \sum_{i =1}^k
\Gamma_\ge (u_{1,\ge}^i) \Big ) &=  \limsup_{\ge \to 0} \Big ( \Gamma_\ge
(u_{2,\ge}) +
 \Gamma_\ge (u_{1,\ge}) \Big ) \\
&\leq  \limsup_{\ge \to 0}\Gamma_\ge (u_\ge) \leq E = \sum_{i=1}^k
E_{m_i}. \nonumber
\end{align}
Thus, since $\Gamma_\ge (u_{2, \ge}) \geq o(1)$ we deduce from
(\ref{200})-(\ref{412}) that, for any $ i \in \{1, \ldots k \}$
\begin{equation}
\label{413} \lim_{\ge \to 0} \Gamma_\ge (u_{1, \ge}^i) = E_{m_i}.
\end{equation}
Now (\ref{200}), (\ref{413}) implies that $L_{V(x^i)}(\omega^i) =
E_{m_i}$. Recalling from  \cite{JT2} that $E_a >E_b$ if $a >b$ and
using Lemma \ref{levels} we conclude that $x^i \in \M^i$. At this
point it is clear that $W^i(y) = e^{i A(x^i)y}U_i(y - z_i)$ with
$U_i \in S_{m_i}$ and $z_i \in \R^3.$

To establish that $W_\ge^i \to W^i$ strongly in $H_\ge$ we first show
that $W_\ge^i \to W^i$ strongly in $L^2(\R^3, \C)$. Since $(W_\ge^i)$
is bounded in $H_\ge$ the Diamagnetic inequality (\ref{eq:2.6})
immediately yields that $(|W_\ge^i|)$ is bounded in $H^1(\R^3, \R)$
and we can assume that $|W_\ge^i| \to |W^i| = |\omega^i|$ weakly in
$H^1(\R^3, \R)$.
Now since $L_{V(x^i)}(\omega^i) = E_{m_i}$, we get
using the Diamagnetic inequality, (\ref{3.101}), (\ref{413}) and the
fact that $V \geq V(x^i)$ on $O^i$,
\begin{align}\label{sopxxxx}
\int_{\R^3} | \nabla \omega^i |^2 dy &+ \int_{\R^3} m_i
|\omega^i|^2 dy -2 \int_{\R^3 \times \R^3} W(x-y) |\omega^i(x)|^2 |\omega^i(y)|^2\, dx dy  \notag \\
&\geq
\limsup_{\ge \to 0}  \int_{\R^3} \left| \left(\frac{\nabla}{\iu} -
A(\ge y + x^i_\ge) \right)
\mathcal{W}_\ge^i \right|^2 dy + \int_{\R^3} V(\ge y + x^i_\ge) |\mathcal{W}_\ge^i|^2  dy \notag \\
& \qquad {}- 2 \int_{\R^3 \times \R^3} W(x-y) |\mathcal{W}_\ge^i(x)|^2 |\mathcal{W}_\ge^i(y)|^2\, dx dy
 \notag \\
&\geq \limsup_{\ge \to 0}  \int_{\R^3} \big| \nabla | \mathcal{W}_\ge^i|
\big|^2 dy + \int_{\R^3} V(x^i) |\mathcal{W}_\ge^i|^2dy \nonumber \\
& {} - 2 \int_{\R^3 \times \R^3} W(x-y) |\mathcal{W}_\ge^i(x)|^2 |\mathcal{W}_\ge^i(y)|^2\, dx dy
 \notag \\
&\geq  \int_{\R^3} \big| \nabla |\omega^i| \big|^2 dy + \int_{\R^3}
m_i |\omega^i|^2 dy \nonumber \\
&{}- 2 \int_{\R^3 \times \R^3} W(x-y) |\omega^i(x)|^2 |\omega^i(y)|^2\, dx dy .
\end{align}
But from Lemma \ref{levels} we know that, since
$L_{V(x^i)}(\omega^i) = E_{m_i}$,
\[
\int_{\R^3} \big| \nabla |\omega^i| \big|^2 dy = \int_{\R^3} \big| \nabla \omega^i \big|^2
dy.
\]
Thus we deduce from \eqref{sopxxxx} that
\begin{equation}\label{1000}
\int_{\R^3}V(\ge y + x^i_\ge) |\mathcal{W}_\ge^i|^2 dy \to \int_{\R^3} V(x^i)
|\mathcal{W}^i|^2 dy.
\end{equation}
Thus, since $V \geq V(x^i)$ on $O^i$, we deduce that
\begin{equation} \label{strongly}
\text{$\mathcal{W}_\ge^i \to \mathcal{W}^i$  strongly in $L^2(\R^3, \C)$}.
\end{equation}
From \eqref{strongly} we easily get that
\begin{equation} \label{strongly0}
\lim_{\ge \to 0}   \int_{\R^3} \left| \left(\frac{\nabla}{\iu} -
A(\ge y + x^i_\ge)\right) \mathcal{W}_\ge^i \right|^2 - \left|
\left(\frac{\nabla}{\iu} - A( x^i)\right) \mathcal{W}_\ge^i \right|^2 dy = 0.
\end{equation}
Now, using (\ref{3.101}), (\ref{sopxxxx}) and (\ref{1000}), we see
from (\ref{strongly0}) that
\begin{multline}\label{sopxxx}
\int_{\R^3} \left| \left(\frac{\nabla}{\iu} - A(x^i)\right)
\mathcal{W}^i \right|^2 dy + \int_{\R^3} V(x^i) |\mathcal{W}^i|^2 dy  \\
\geq \limsup_{\ge \to 0}  \int_{\R^3} \left|
\left(\frac{\nabla}{\iu} - A(\ge y  + x^i_\ge)\right)
\mathcal{W}_\ge^i \right|^2 dy + \int_{\R^3} V(\ge y + x^i_\ge) |\mathcal{W}_\ge^i|^2  dy  \\
\geq \limsup_{\ge \to 0}  \int_{\R^3} \left| \left(\frac{\nabla}{\iu}
- A( x^i)\right) \mathcal{W}_\ge^i \right|^2 dy + \int_{\R^3} V(x^i) |\mathcal{W}_\ge^i|^2
\, dy.
\end{multline}
At this point and using Corollary \ref{equiv-norm-consequence} ii)
we have established the strong convergence $W_\ge^i \to W^i$ in
$H^1(\R^3, \C)$. Thus we have
\[
u_{1, \ge}^i = e^{iA(x^i)(\cdot - x^i_\ge / \ge) } U_i(\cdot - x^i_\ge /
\ge - z_i) + o(1)
\]
strongly in $H^1(\R^3, \C)$. Now setting $y_\ge^i = x^i_\ge/\ge +z_i$
and changing $U_i$ to $e^{i A(x^i) z_i}U_i$ we get that
\[
u_{1, \ge}^i = e^{i A(x^i) (\cdot - y_\ge^i)} U_i(\cdot - y_\ge^i) +
o(1)
\]
strongly in $H^1(\R^3, \C)$. Finally using the exponential decay of
$U_i$ and $\nabla U_i$ we have
\[
u_{1, \ge}^i = e^{i A_\ge(y_\ge^i) (\cdot - y_\ge^i)} \p_{\ge}(\cdot -
y_\ge^i)U_i(\cdot - y_\ge^i) + o(1).
\]
From Corollary~\ref{equiv-norm-consequence} iii) we deduce that this
convergence also holds in $H_\ge$ and thus
\[
u_{1,\ge} = \sum_{i=1}^k u_{1, \ge}^i = \sum_{i=1}^k  e^{i
A_\ge(y_\ge
^i) (\cdot - y_\ge^i)} \p_{\ge}(\cdot - y_\ge^i)U_i(\cdot -
y_\ge^i) + o(1)
\]
strongly in $H_\ge$. To conclude the proof of the Proposition, it
suffices to show that $u_{2, \ge} \to 0$ in $H_\ge$. Since $E \geq
\lim_{\ge \to 0} \Gamma_\ge(u_\ge)$ and $\lim_{\ge \to 0}
\Gamma_\ge(u_{1, \ge})= E$ we deduce, using (\ref{402}) that $\lim_{\ge
\to 0} \Gamma_\ge(u_{2,\ge})= 0$. Now from (\ref{405})-(\ref{407}) we
get that $u_{2, \ge} \to 0$ in $H_\ge$.
\end{proof}

\subsection{Critical points of the penalized functional}

We first state the following

\begin{proposition}\label{prop4}
For sufficiently small $d> 0,$ there exist constants $\omega > 0$
and $\e_0
> 0$ such that $|\Gamma_\e^\prime(u)| \geq \omega$ for $u \in
\Gamma^{D_{\e}}_\e \cap (X_\e^{d} \setminus X_\e^{d/2})$ and $\e \in
(0,\e_0).$
\end{proposition}
\begin{proof} By contradiction, we suppose that for
$d > 0$ sufficiently small such that Proposition \ref{prop40}
applies, there exist $(\e_j)$ with $\lim_{j \to \infty}\e_j = 0$ and
a sequence $(u_{\e_j})$ with $u_{\e_j} \in X_{\e_j}^{d} \setminus
X_{\e_j}^{d/2}$ satisfying $\lim_{j \to
\infty}\Gamma_{\e_j}(u_{\e_j}) \le E$ and $\lim_{j \to
\infty}\Gamma'_{\e_j}(u_{\e_j}) = 0.$  By Proposition
\ref{prop40}, there exist $(y^i_{\e_j}) \subset \R^3$, $i
=1,\ldots,k,$ $x^i\in \M^i$, $U_i \in S_{m_i}$ such that
\begin{equation*}
\lim_{\e_j \to 0} |\e_j y^i_{\e_j} - x^i| = 0,
\end{equation*}
\begin{equation*}
\lim_{\e_j \to 0} \Big\Vert u_{\varepsilon_j} - \sum_{i=1}^k
e^{iA_{\e_j}(y^i_{\e_j}) (\cdot - y^i_{\e_j})} \p_{\e_j}(\cdot -
y^i_{\e_j}) U_i(\cdot - y^i_{\e_j}) \Big\Vert_{\e_j} = 0.
\end{equation*}
By definition of $X_{\e_j}$ we see that $\lim_{\e_j \to
0}\operatorname{dist}(u_{\e_j},X_{\e_j}) = 0.$ This contradicts that
$u_{\ge_j} \not \in X_{\ge_j}^{d/2}$ and completes the proof.
\end{proof}

From now on we fix a $d>0$ such that Proposition \ref{prop4} holds.

\begin{proposition}\label{prop7}
For sufficiently small fixed $\e >  0,$  $\Gamma_\e$ has a critical
point $u_{\e} \in X_\e^{d} \cap \Gamma_{\e}^{D_{\e}}.$
\end{proposition}

\begin{proof}
We can take $R_0>0$ sufficiently large so that $O \subset B(0,R_0)$
and $\gamma_\varepsilon(s) \in H^1_0(B(0, R/\varepsilon))$ for any
$s \in T$, $R > R_0$ and sufficiently small $\varepsilon
>0$.

We notice that by Proposition \ref{prop2} (iii), there exists
$\alpha \in (0, E- \tilde{E})$ such that for sufficiently small $\e
> 0$,
\[
\Gamma_\e (\gamma_\e(s)) \geq D_\e - \alpha \quad \Longrightarrow
\quad \gamma_\e(s) \in X_{\e}^{d/2} \cap H^1_0(B(0,R/\varepsilon)).
\]
We begin to show that for sufficiently small fixed $\e > 0$, and
$R>R_0$, there exists a sequence $(u_n^R) \subset X_{\e}^{d/2} \cap
\Gamma_{\e}^{D_{\e}} \cap H^1_0(B(0,R/\varepsilon))$ such that
$\Gamma'(u_n^R) \to 0$ in $H^1_0 (B(0,R/\varepsilon))$ as $n \to +
\infty$.

Arguing by contradiction, we suppose that for sufficiently small $\e
> 0,$ there exists $a_R(\e) >0$ such that $|\Gamma_\e'(u)| \geq a_R(\e)$
on $X_\e^{d} \cap \Gamma_{\e}^{D_\e} \cap H^1_0
(B(0,R/\varepsilon))$. In what follows any $u \in H^1_0
(B(0,R/\varepsilon))$ will be regarded as an element in
$H_\varepsilon$ by defining $u=0$ in $\R^3 \setminus
B(0,R/\varepsilon)$. Note from Proposition \ref{prop4} that there
exists $\omega > 0,$ independent of $\e > 0,$ such that
$|\Gamma_\e^\prime(u)| \geq \omega$ for $u \in \Gamma_{\e}^{D_{\e}}
\cap (X_\e^{d} \setminus X_\e^{d/2}).$ Thus, by a deformation
argument in $H^1_0 (B(0,R/\varepsilon))$, starting from $\gamma_\e$,
for sufficiently small $\e
>0$ there exists a $\mu \in (0,\alpha)$ and a path $\gamma \in
C([0, T], H_\e)$ satisfying
\[
\gamma(s) = \gamma_\e(s) \quad\text{for $\gamma_\e(s) \in
\Gamma_\e^{D_\e - \alpha}$}, \quad \gamma(s) \in  X_\e^{d} \quad
\text{ for $\gamma_\e(s) \notin \Gamma_\e^{D_\e - \alpha}$}
\]
and
\begin{equation} \label{61}
\Gamma_\e(\gamma(s)) < D_\e - \mu, \quad s \in T.
\end{equation}

Let $\psi \in C_0^\infty(\RN)$ be such that $\psi(y) = 1$ for $y \in
O^{\delta},$ $\psi(y) = 0$ for $y \notin O^{2\delta},$ $\psi(y) \in
[0,1]$ and $|\nabla \psi| \le 2/\delta.$ For $\gamma(s) \in X_\e^d,$
we define $\gamma_1(s) = \psi_\e\gamma(s)$ and $\gamma_2(s) =
(1-\psi_\e)\gamma(s)$ where $\psi_\e (y) = \psi(\e y)$. The
dependence on $\varepsilon$ will be understood in the notation for
$\gamma_1$ and $\gamma_2$.
 Note that
\begin{align*}
\Gamma_\e(\gamma(s))  &=  \Gamma_\e(\gamma_1(s)) +
\Gamma_\e(\gamma_2(s)) +
\intr \bigl(\psi_\e(1-\psi_\e)|D^\e \gamma(s)|^2 +  V_\e \psi_\e(1-\psi_\e)|\gamma(s)|^2 \bigr)dy \\
&\quad{} + Q_\e(\gamma(s)) - Q_\e(\gamma_1(s)) - Q_\e(\gamma_2(s))  \\
&\quad {} - \frac{1}{4} \int_{\R^3 \times \R^3} W(x-y)
\bigl(|\gamma(s)(x)|^2 |\gamma(s)(y)|^2 -|\gamma_1(s)(x)|^2
|\gamma_1(s)(y)|^2 \\
&\quad \quad{}- |\gamma_1(s)(x)|^2 |\gamma_2(s)(y)|^2 \bigr) dx
\, dy + o(1).
\end{align*}
Since for $A,B \geq 0,$ $(A+B-1)_+ \geq (A-1)_+ + (B-1)_+$, it
follows that
\begin{align*} Q_\e(\gamma(s))
 & =   \Big (\intr \chi_\e|\gamma_1(s)+\gamma_2(s)|^2 dy - 1\Big )_+^{\frac{5}{2}} \\
 & \geq   \Big (\intr \chi_\e|\gamma_1(s)|^2 dy + \intr \chi_\e |\gamma_2(s)|^2 dy - 1\Big )_+^{\frac{5}{2}} \\
 & \geq  \Big (\intr \chi_\e|\gamma_1(s)|^2 dy - 1\Big
)_+^{\frac{5}{2}} +
 \Big (\intr \chi_\e|\gamma_2(s)|^2 dy - 1\Big )_+^{\frac{5}{2}}\\
 & =  Q_\e(\gamma_1(s)) + Q_\e(\gamma_2(s)).
\end{align*}
Now, as in the derivation of (\ref{407}), using the fact that
$Q_\e(\gamma(s))$ is uniformly bounded with respect to $\varepsilon$, we have, for some $C>0$
\begin{equation}
\label{4007} \int_{\RN \setminus O_\e} |\gamma(s)|^2 dy \le C\e^{6 /
\mu}.
\end{equation}
Since $W$ is even, we have
\begin{multline*}
\int\limits_{\R^3 \times \R^3} W(x-y)
\bigl(|\gamma(s)(x)|^2
|\gamma(s)(y)|^2 -|\gamma_1(s)(x)|^2 |\gamma_1(s)(y)|^2 -
|\gamma_1(s)(x)|^2 |\gamma_2(s)(y)|^2 \bigr) dx dy  \\
=2 \int_{O^{ \delta}_\e}  dy \int_{\R^3 \setminus O^{2 \delta}_\e}
 W(x-y) |\gamma(s)(x)|^2 |\gamma(s)(y)|^2 dx
\\ {} + 2
\int_{O^{\delta}_\e} dy  \int_{O^{2 \delta}_\e  \setminus O^{
\delta}_\e}  W(x-y) |\gamma(s)(x)|^2 |\gamma(s)(y)|^2 dx
\\ +
2 \int_{O^{2 \delta}_\e \setminus O^{\delta}_\e} dy \int_{O^{2
\delta}_\e  \setminus O^{\delta}_\e}  W(x-y) |\gamma(s)(x)|^2
|\gamma(s)(y)|^2 dx
\\ +
2 \int_{\R^3 \setminus O^{2 \delta}_\e } dy \int_{O^{2 \delta}_\e
\setminus O^{ \delta}_\e}  W(x-y) |\gamma(s)(x)|^2 |\gamma(s)(y)|^2
dx  \\
= 2 \int_{O^{\delta}_\e}  dy \int_{\R^3 \setminus O^{\delta}_\e}
 W(x-y) |\gamma_2(s)(x)|^2 |\gamma_1(s)(y)|^2 dx
\\ {} + 2
 \int_{\R^3 \setminus O_\e^\delta} dy \int_{O^{2
\delta}_\e  \setminus O^{\delta}_\e}  W(x-y) |\gamma(s)(x)|^2
|\gamma(s)(y)|^2 dx
\end{multline*}
From (\ref{4007}) we deduce that
\begin{equation}\label{tior}
\lim_{\e \to 0} \int_{O^{\delta}_\e}  dy \int_{\R^3 \setminus
O^{\delta}_\e}
 W(x-y) |\gamma_2(s)(x)|^2 |\gamma_1(s)(y)|^2 dx = 0
\end{equation}
and
\begin{equation}\label{tior1}
\lim_{\e \to 0} \int_{\R^3 \setminus O_\e^\delta}  dy \int_{O^{2
\delta}_\e  \setminus O^{\delta}_\e}
 W(x-y) |\gamma_2(s)(x)|^2 |\gamma_1(s)(y)|^2 dx = 0
\end{equation}

From (\ref{tior}) and (\ref{tior1}) we have (recall that $\gamma_1$ and $\gamma_2$ depend on $\varepsilon$)
\begin{multline*}
\int\limits_{\R^3 \times \R^3} \bigl|W(x-y) \bigl(|\gamma(s)(x)|^2
|\gamma(s)(y)|^2 -|\gamma_1(s)(x)|^2 |\gamma_1(s)(y)|^2 -
|\gamma_1(s)(x)|^2 |\gamma_2(s)(y)|^2 \bigr) \bigr| dx dy  \\
=o(1).
\end{multline*}
  Thus, we see that, as $\e \to 0$,
\[
\Gamma_\e(\gamma(s)) \geq \Gamma_\e(\gamma_1(s)) +
\Gamma_\e(\gamma_2(s)) + o(1).
\]
 Also
\[
\Gamma_\e(\gamma_2(s)) \geq - \frac{1}{4} \int_{(\R^3 \setminus O_\e) \times (\R^3 \setminus O_\e)}
W(x-y) |\gamma_2(s)(x)|^2 |\gamma_2(s)(y)|^2 \, dx\, dy  \geq o(1).
\]
Therefore it follows that
\begin{equation}
\label{62} \Gamma_\e(\gamma(s)) \geq  \Gamma_\e(\gamma_1(s)) + o(1).
\end{equation} \smallskip

For $i = 1,\ldots,k,$ we define
\[
\gamma_1^{i}(s)(y) =
\begin{cases}
\gamma_1(s)(y) &\text{for $y\in$ $(O^i)^{2\delta}_\e$} \\
0 &\text{for $y \notin (O^i)^{2\delta}_\e$}.
\end{cases}
\]
 Note
that $(A_1 +\cdots +A_n - 1)_+ \geq \sum_{i=1}^n(A_i-1)_+$ for
$A_1,\ldots,A_n \geq 0$. Then we see that
\begin{equation}\label{63}
\Gamma_\e(\gamma_1(s)) \geq
\sum_{i=1}^k\Gamma_\e(\gamma_1^{i}(s))
 = \sum_{i=1}^k\Gamma^i_\e(\gamma_1^{i}(s)).
\end{equation}
From Proposition \ref{prop2} (ii) and since $0<\alpha <E- \tilde{E}$
we get that $\gamma_1^{i} \in \Phi_\e^i$, for all $i \in \{1,
\ldots, k\}$. Thus by Proposition 3.4 in \cite{CR1}, Proposition
\ref{upper}, and (\ref{63}) we deduce that, as $\e \to 0$,
\[
\max_{s \in T}\Gamma_\e(\gamma(s)) \geq E + o(1).
\]
Since $\limsup_{\e \to 0}D_\e \le E$ this contradicts~\eqref{61}.

Now let $(u_n^R)$ be a Palais-Smale sequence corresponding to a
fixed small $\e > 0$. Since $(u_n^R)$ is bounded in
$H^1_0(B(0,R/\epsilon)),$ and by Corollary
\ref{equiv-norm-consequence}, we have that, up to subsequence,$u_n^R
$ converges strongly to $u^R$ in $H^1_0(B(0,R/\epsilon))$.
We observe that $u^R$ is a critical point of $\Gamma_\varepsilon$ on
$H^1_0(B(0,R/\epsilon)),$ and it solves
\begin{multline} \label{onballs}
\bigg( \frac{1}{\iu} \nabla - A_\e \bigg)^2  u^R + V_\e u^R \\
= \left( W
* |u^R|^2 \right) u^R  - 5 \, \Big(\int\chi_\e
|u^R|^2 dy - 1\Big)_+^{\frac{3}{2}}\chi_\e u^R \ \textrm{ in } \
B(0, R/\epsilon).
\end{multline}
Exploiting Kato's inequality,
\[
\Delta |u^R| \geq - \re \bigg( \frac{\bar{u^R}}{|u^R|}\bigg(
\frac{\nabla}{\iu} - A_\e \bigg)^2 u^R\bigg)
\]
we obtain
\begin{equation} \label{katoine}
\Delta |u^R| \geq V_\e |u^R| - \left( W * |u^R|^2 \right) |u^R|+ 5
\,\Big(\int\chi_\e |u^R|^2 dy - 1\Big)_+^{\frac{3}{2}}\chi_\e |u^R|
\ \textrm{ in } \ \RN.
\end{equation}
Moreover by Moser iteration it follows that $\|u^R\|_{L^\infty}$ is
bounded. Since $(Q_\epsilon(u^R))$ is uniformly bounded for
$\epsilon>0$ small, we derive that $ \left( W * |u^R|^2 \right)
|u^R| \leq \frac{1}{2} V_\ge |u^R(y)|$ if $|y| \geq 2 R_0$. Applying
a comparison principle we derive that
\begin{equation}\label{decay}
|u^R(y)| \leq C \exp(-(|y| - 2 R_0))
\end{equation}
for some $C>0$ independent of $R > R_0$. Therefore as $(u^R)$ is
bounded in $H_\epsilon$ we may assume that it weakly converges to
some $u_\epsilon$ in $H_\epsilon$ as $R \to + \infty$. Since $u^R$
is a solution of $(\ref{onballs})$, we see from $(\ref{decay})$ that
$(u^R)$ converges strongly to $u_\epsilon \in X_\epsilon \cap
\Gamma^{D_\epsilon}_\epsilon$ and it solves
\begin{equation} \label{101}
\bigg( \frac{1}{\iu} \nabla - A_\e \bigg)^2  u_\e + V_\e u_\e = \left(
W * |u_\varepsilon|^2 \right) u_\varepsilon  - 5 \,\Big(\int\chi_\e
|u_\e|^2 dy - 1\Big)_+^{\frac{3}{2}}\chi_\e u_\e \ \textrm{in $\RN$}.
\end{equation}
\end{proof}

\subsection{Proof for the main result}

We see from Proposition \ref{prop7} that there exists $\e_0 > 0$ such
that, for $\e \in (0,\e_0),$ $\Gamma_\e$ has a critical point $u_\e
\in X_\e^d \cap \Gamma_\e^{D_\e}.$  Exploiting Kato's inequality
\[
\Delta |u_\e| \geq - \re \bigg( \frac{\bar{u_\e}}{|u_\e|}\bigg(
\frac{\nabla}{\iu} - A_\e \bigg)^2 u_\e\bigg)
\]
we obtain
\begin{equation} \label{102}
\Delta |u_\e| \geq V_\e |u_\e| - \left( W * |u_\varepsilon|^2
\right) |u_\varepsilon |+ 5 \Big(\int\chi_\e |u_\e|^2 dy -
1\Big)_+^{\frac{3}{2}}\chi_\e |u_\e| \ \textrm{ in } \ \RN.
\end{equation}
Moreover, by (\ref{eq:tutti_q}) and the subsequent bootstrap arguments, we deduce that $u_\ge \in L^q(\mathbb{R}^3)$ for any $q>2$. Hence a Moser iteration scheme shows that $(\| u_\e
\|_{L^\infty})$ is bounded. Now by Proposition \ref{prop40}, we see that
\[
\lim_{\e \to 0} \int_{\RN \setminus (\M^{2\beta})_\e}|D^\e u_\e|^2
+ \tilde V_\e|u_\e|^2 dy = 0,
\]
and thus, by elliptic estimates (see \cite{GT}), we obtain that
\begin{equation} \label{103}
\lim_{\e \to 0} \Vert  u_\e  \Vert_{L^\infty(\RN \setminus
(\M^{2\beta})_\e)} = 0. \end{equation} This gives the following
decay estimate for $u_\e$ on $\R^3 \setminus (\M^{2\beta})_\e \cup
(\Z^\beta)_\e$
\begin{equation} \label{104}
 |u_\e(x)| \le C \exp(-c \operatorname{dist}(x, (\M^{2\beta})_\e \cup
(\Z^{\beta})_\e))
\end{equation}
for some constants $C, c >0$. Indeed from
(\ref{103}) we see that
\[
\lim_{\e \to 0}   \| W * |u_\ge|^2\|_{L^\infty (\R^3 \setminus
(\M^{2\beta})_\e \cup (\Z^{\beta})_\e)} =0.
\]
Also $\inf \{V_\e(y) : y \notin (\M^{2\beta})_\e \cup
(\Z^{\beta})_\e \} >0$. Thus, we obtain the decay estimate
(\ref{104}) by applying standard comparison principles to (\ref{102}). \medskip

If $\Z \neq \emptyset$ we shall need, in addition, an estimate for
$|u_\e|$ on $(\Z^{2\beta})_\e$. Let $\{H^i\}_{i \in I}$ be the
connected components of int$(\Z^{3\delta})$ for some index set
$I.$ Note that $\Z \subset \bigcup_{i \in I}H^i$ and $\Z$ is
compact. Thus, the set $I$ is finite. For each $i \in I,$ let
$(\phi^i, \lambda^i_1)$ be a pair of first positive eigenfunction
and eigenvalue of $- \Delta$ on $(H^i)_\e$ with Dirichlet boundary
condition. From now we fix an arbitrary $i \in I$. By using the fact that
$(Q_\e(u_\e))$ is bounded we see that for some constant $C>0$
\begin{equation} \label{l-a}
 \Vert  u_\e  \Vert_{L^3 ((H^i)_\e)} \le C  \e^{3 / \mu}.
\end{equation}
Thus, from the Hardy-Littlewood-Sobolev inequality we have, for some $C>0$
\[
\|W*
|u_\e|^2 \|_{L^\infty ((H^i)_\e)} \leq C \|u_\varepsilon\|^2_{L^3 ((H^i)_\e)} \leq C \e^6.
\]
Denote $\phi^i_\e(y) = \phi^i(\e y)$. Then, for sufficiently small
$\e > 0$, we deduce that for $y \in \textrm{int}((H^i)_\e)$,
\begin{equation} \label{l-b}
\Delta \phi^i_\e (y) -V_\e(x)\phi^i_\e(y) + (W*|u_\e(y)|^2)\phi^i_\e
(y) \le \Big ( C \e^6 - \lambda_1 \e^2 \Big) \phi^i_\e \leq 0 .
\end{equation}
Now, since $\operatorname{dist}(\partial (\Z^{2\beta})_\e,
(\Z^{\beta })_\e) = \beta /\e$, we see from (\ref{104}) that for
some constants $C, c
>0$,
\begin{equation} \label{l-c}
\|u_\e\|_{L^{\infty} (\partial (\Z^{2 \beta})_\e)} \le C \exp (- c
/\e).
\end{equation}

We normalize $\phi^i$ requiring
that
\begin{equation} \label{l-d}
\inf_{y \in (H^i)_\e \cap
\partial (\Z^{2 \delta})_\e} \phi^i_\e(y) = C \exp( -c /\e)
\end{equation}
for the same  $C$, $c>0$ as in (\ref{l-c}). Then, we see that for some $\kappa  > 0$,
\[
\phi^i_\e(y) \le \kappa C \exp(-c/\e), \quad y \in (H^i)_\e \cap
(\Z^{2\beta})_\e.
\]
Now we deduce, using (\ref{l-a}), (\ref{l-b}), (\ref{l-c}),
(\ref{l-d}) that for each $i \in I,$
$|u_\e| \leq \phi^i_\e$ on $(H^i)_\e \cap (\Z^{2 \beta})_\e$.
Therefore
\begin{equation} \label{105}
|u_\e (y)| \leq C \exp (- c /\e),\,\, \mbox{ on } (\Z^{2 \delta})_\e
\end{equation}
for some $C, c > 0$. Now (\ref{104}) and (\ref{105}) implies
 that $Q_\e(u_\e) = 0$ for $\e >0$ sufficiently small and thus $u_\e$ satisfies (\ref{eq:Hartree-scaled}).
Now using Propositions~\ref{prop:compact} and~\ref{prop40}, we readily
deduce that the properties of $u_\e$ given in Theorem \ref{main}
hold. Here, in (\ref{eq:1.5}) we also use
Lemma \ref{levels}. $\Box$

\bigskip

\end{document}